\documentclass{amsart}
\usepackage{amsmath,amssymb,amsfonts}
\usepackage{dsfont}
\usepackage{stmaryrd}
\usepackage{url}
\usepackage{hyperref}
\usepackage{tikz} \usetikzlibrary{patterns}
\definecolor{bleuf}{rgb}{0.,0.,0.8}
\usepackage[normalem]{ulem}
\newcounter{corr}
\definecolor{violet}{rgb}{0.580,0.,0.827}
\newcommand{\corr}[3]{\typeout{Warning : a correction remains in page
\thepage}
				\stepcounter{corr}   \def\identite{{\mathrm{Id}}}     
				{\color{blue}\ifmmode\text{\,\sout{\ensuremath{#1}}\,}\else\sout{#1}\fi}
       {\color{red}#2}
       {\color{violet} #3}}

\newtheorem{theorem}{Theorem}[section]
\newtheorem{definition}{Definition}[section]
\newtheorem{lemma}[theorem]{Lemma}
\newtheorem{remark}{Remark}[section]

\newcommand{\bfe}{{\boldsymbol e}}
\newcommand{\bfE}{{\boldsymbol E}}
\newcommand{\bff}{{\boldsymbol f}}

\newcommand{\bfn}{\boldsymbol n}

\newcommand{\bfu}{{\boldsymbol u}}
\newcommand{\bfv}{{\boldsymbol v}}
\newcommand{\bfV}{{\boldsymbol V}}
\newcommand{\bfw}{{\boldsymbol w}}
\newcommand{\bfW}{{\boldsymbol W}}
\newcommand{\bfx}{\boldsymbol x}

\newcommand{\bfvarphi}{{\boldsymbol \varphi}}
\newcommand{\bfH}{{\boldsymbol H}}

\newcommand{\bfC}{\boldsymbol C}

\newcommand{\dx}{\ \mathrm{d}\bfx}
\newcommand{\dt}{\ \mathrm{d} t}
\newcommand{\ds}{\ \mathrm{d} s}

\newcommand{\nstep}{N}
\newcommand{\nalgo}{n}
\newcommand{\nnn}{{n \in \xN}}
\newcommand{\nti}{n \to + \infty}
\newcommand{\Nti}{\nstep \to + \infty}
\newcommand{\deltat}{{\delta t}}
\newcommand{\disc}{{\mathcal D}}
\newcommand{\mesh}{{\mathcal M}}
\newcommand{\edge}{{\sigma}}

\newcommand{\edges}{{\mathcal E}}

\newcommand{\edgesext}{{\mathcal E}_{\mathrm{ext}}}

\newcommand{\edgesexti}{{\mathcal E}_{\mathrm{ext}}^{(i)}}
\newcommand{\edgesi}{{\edges\ei}}
\newcommand{\edgesj}{{\edges\ej}}
\newcommand{\edged}{\epsilon}
\newcommand{\edgesd}{{\widetilde {\edges}}}

\newcommand{\edgesdinti}{{\edgesd^{(i)}_{{\rm int}}}}
\newcommand{\edgesdexti}{{\edgesd^{(i)}_{{\rm ext}}}}
\newcommand{\ei}{^{(i)}}
\newcommand{\ej}{^{(j)}}

\newcommand{\Hmeshzero}{\bfH_{\! N,0}}

\newcommand{\HmeshNzero}{\bfH_{\! N,0}}
\newcommand{\Hmeshi}{H_N^{(i)}}

\newcommand{\HmeshiNzero}{H_{N,0}^{(i)}}

\newcommand{\dive}{{\mathrm{div}}}
\newcommand{\gradi}{\boldsymbol \nabla}

\newcommand{\xN}{\mathbb{N}}
\newcommand{\xR}{\mathbb{R}}
 
\newcommand{\blist}{\begin{list}{-}{\itemsep=0.5ex \topsep=0.5ex \leftmargin=1.cm \labelwidth=0.3cm \labelsep=0.5cm \itemindent=0.cm}}

\newcommand{\Char}{{\mathds 1}}

\newcounter{cst}
\newcommand{\ctel}[1]{C_{\refstepcounter{cst}\label{#1}\thecst}}
\newcommand{\cter}[1]{C_{\ref{#1}}}

\begin{document}
\title[Convergence of the  projection scheme]
{Convergence of the fully discrete incremental projection scheme for  incompressible flows}

\author{T. Gallou\"et}
\address{I2M UMR 7373, Aix-Marseille Universit\'e, CNRS, Ecole Centrale de Marseille. 39 rue Joliot Curie, 13453 Marseille, France}
\email{thierry.gallouet@univ-amu.fr}

\author{R. Herbin}
\address{I2M UMR 7373, Aix-Marseille Universit\'e, CNRS, Ecole Centrale de Marseille. 39 rue Joliot Curie, 13453 Marseille, France}
\email{raphaele.herbin@univ-amu.fr}

\author{J.C. Latch\'e}
\address{Institut de Radioprotection et de S\^{u}ret\'{e} Nucl\'{e}aire (IRSN)}
\email{jean-claude.latche@irsn.fr}

\author{D. Maltese}
\address{LAMA, Universit\'e Gustave Eiffel, France.}
\email{david.maltese@univ-e}

\subjclass[2000]{35Q30, 65M08, 65N12, 76M12}
\keywords{Incompressible Navier-Stokes equations, projection scheme, staggered finite volumes, MAC scheme, analysis, convergence}

\begin{abstract}
The present paper addresses the convergence of a first order in time incremental projection scheme for the time-dependent incompressible Navier--Stokes equations to a weak solution, without any assumption of existence or regularity assumptions on the exact solution. 
We prove  the convergence of the approximate solutions obtained by the semi-discrete scheme and a fully discrete scheme using a  staggered finite volume scheme on non uniform rectangular meshes. 
Some first \emph{a priori} estimates on the approximate solutions yield the existence.
Compactness arguments, relying on these estimates, together with some estimates on the translates of the discrete time derivatives, are then developed to obtain convergence (up to the extraction of a subsequence), when the time step tends to zero in the semi-discrete scheme and when the space and time steps tend to zero in the fully discrete scheme;  the approximate solutions are thus shown to converge to a limit function which  is then shown to be a weak solution to the continuous problem by passing to the limit in these schemes.
\end{abstract}
\maketitle
%
%
\section{Introduction}
The incompressible Navier–Stokes equations for a homogeneous fluid read:
\begin{subequations} \label{pb:cont}
\begin{align}
 \label{qdm} &
 \partial_t \bfu + (\bfu \cdot \gradi)\bfu -  \Delta \bfu +\nabla p = \bff~\text{in}~(0, T) \times \Omega,
\\ \label{inc} &
\dive \bfu=0~\text{in}~(0, T) \times \Omega,
\end{align}
\end{subequations}
where the density and the viscosity are set to one for the sake of simplicity, and  where
\begin{equation}\label{hyp:T-Omega}
    \begin{array}{l}
        T>0, \mbox{ and }\Omega  \mbox{ is a connected, open and bounded subset of }\ \xR^3,\\  \mbox{with a Lipschitz boundary } \partial \Omega.
    \end{array}
\end{equation}
Note that we only consider the three dimensional setting in this work, but the analysis may be carried out in a similar (and often somewhat simpler) manner in the one or two dimensional setting. 
The variables $\bfu$ and $p$ are respectively the velocity and the pressure in the flow, and Eqns. \eqref{qdm} and \eqref{inc} respectively enforce the momentum conservation and the mass conservation and incompressibility of the flow. 
This system is supplemented with the boundary condition 
\begin{equation}
\bfu=0~\text{on}~ (0,T) \times \partial \Omega,  \label{boundary}
\end{equation}%
and the initial condition 
\begin{equation}
\bfu(0)=\bfu_0~\text{in}~  \Omega.  \label{init}
\end{equation} 
The function $\bfu_0$ is the initial datum for the velocity and the function $\bff$ is the source term.
Throughout the paper, we shall assume that 
\begin{equation}\label{hyp:f-u0}
\bff 
\in L^2((0,T) \times \Omega)^3  \mbox{ and }\bfu_0 
  \in \bfE(\Omega),
\end{equation}
where $\bfE(\Omega)$ is the subset of $H_0^1(\Omega)^3$ of divergence-free functions, defined by
\[
\bfE(\Omega) = \{ \bfu \in H_0^1(\Omega)^3~\text{such that}~\dive \bfu = 0 \}.
\]
Note that in fact, the initial condition is assumed to be in $\bfE(\Omega)$ for the sake of simplicity. 
It could be considered in $L^2(\Omega)^3$ only, see Remark \ref{rem:semid-init}.

Let us define the weak solutions of Problem  ($\ref{pb:cont}$)-($\ref{init}$) in the sense of Leray \cite{leray1934}.
\begin{definition}[Weak solution]\label{def:weaksol}
Under the assumptions \eqref{hyp:T-Omega} and \eqref{hyp:f-u0}, a function $\bfu \in L^2(0,T;\bfE(\Omega))$ $\cap$ $L^\infty(0,T;L^2(\Omega)^3)$ is a weak solution of the  problem \eqref{pb:cont}-\eqref{init} if 
\begin{multline} \label{weaksol}
- \int_0^T \!\!\int_\Omega \bfu  \cdot \partial_t \bfv  \dx - \int_0^T \!\!\int_\Omega \bfu \otimes \bfu : \gradi \bfv \dx \dt   + \int_0^T \!\!\int_\Omega \gradi \bfu  : \gradi \bfv \dx \dt \\
= \int_\Omega \bfu_0 \cdot \bfv(0,\cdot) \dx +  \int_0^T \!\!\int_\Omega \bff \cdot \bfv \dx \dt
\end{multline}
for any  $\bfv$ in $\bigl \{\bfw\in C_c^\infty (\Omega \times [0,T))^3,\ \dive \bfw=0 \mbox{ a.e. in } \Omega \times (0,T) \bigr\}$.
\end{definition}

The first projection method to solve the system \eqref{pb:cont} was designed over 50 years ago, and is known as the Chorin-Temam algorithm \cite{Chorin1969OnTC,Temam1969SurLD,temam1984navier}. 
It consists in a prediction step based on a linearized momentum equation without the pressure gradient, and a pressure correction step that enforces the divergence-free constraint.
This method and its variants are now often referred to (following \cite{Guermond2006AnOO}) as non incremental projection schemes, in opposition to the incremental projection schemes that were obtained by adding the old pressure gradient in the prediction step (see \cite{Goda1979AMT} for a first-order time scheme and \cite{Kan1986ASA} of a second order time scheme). 
These latter schemes are indeed  incremental in the sense that the correction step may now be seen as solving an equation on the time increment of the pressure.
They seem to be much more efficient from a computational point of view  \cite{Guermond2006AnOO} and have been the object of several error analysis, under some regularity assumptions on the solution of the continuous problem, in the semi discrete setting, see \cite{Guermond2006AnOO} and references therein. 

The non incremental schemes have been the object of some analyses in the fully discrete setting. 
In \cite{bad-2007-conv} some error estimates are derived for a non incremental scheme with a discretization by the finite element, under some regularity assumptions on the exact solution,
In a recent paper the approximate solutions of a fully discrete non incremental scheme with a uniform staggered discretization \cite{Kuroki2020OnCO} are shown to converge to a weak solution (and so without any regularity assumption on the solution of \eqref{pb:cont}) under the condition that $h \le \delta t^{3-\alpha}$ where $h$ and $\delta t$ are respectively the mesh size and the time step and with $0 < \alpha \le 2$.

However, to our knowledge, up to now, no proof of convergence exists for the fully discrete incremental projection schemes, even though they are the most used in practice.
The purpose of the present work is therefore to fill this gap and to show the convergence of the incremental projection method with a discretization by a staggered finite volume scheme based on a (non uniform) MAC grid, without any regularity assumption on the exact solution.

The Marker-And-Cell (MAC) scheme, introduced in the middle of the sixties (see \cite{Harlow1965NumericalCO}), is one of the most popular methods (see e.g. \cite{Patankar2019NumericalHT} and \cite{Wesseling2000PrinciplesOC}) for the approximation of the Navier --Stokes equations in the engineering framework, because of its simplicity, its efficiency and its remarkable mathematical properties. 
Although originally presented as a finite difference scheme on uniform meshes, the MAC scheme is in fact a finite volume scheme and as such can be used on non uniform meshes. 
The convergence analysis of the staggered finite volume scheme on the MAC mesh using a fully implicit time scheme may be found in \cite{gal-18-convmac}, and we shall use several of the tools developed therein.
We also refer to this latter paper for some more references on studies of the MAC scheme.

 \medskip
The paper is organized as follows.
Section \ref{sec:semidisc} deals with the convergence analysis for the semi-discrete projection algorithm. 
The fully discrete scheme is analysed in Section \ref{sec:MAC-def}; we only give the main ingredients of the staggered space discretization that we use, and which is often referred to as the MAC scheme. 
To avoid a lengthy description, the precise definitions of the now classical discrete MAC operators are to be found in \cite{gal-18-convmac}. 
  
Before starting the analysis of the semi-discrete and fully discrete schemes, we wish to recall, for the sake of clarity, that:
\begin{itemize}
 \item In a Banach space $E$ equipped with a norm $\Vert \cdot \Vert_E$, a sequence $(u_n)_\nnn \subset E$ is said to converge to $u \in E$ if $\Vert u_n - u\Vert_E\to 0$ as $\nti$, while it is said to weakly converge to $u \in E$ if for any continuous linear form $T\in E'$, one has $T(u_n) \to  T(u)$ as $\nti$.
\item A sequence $(T_n)_\nnn \subset E'$ is said to $\star$-weakly converge to $T \in E'$ if for any $u \in E$, one has $T_n(u) \to  T(u)$ as $\nti$.
\item If $E=L^p(\Omega)$, where $1 \le p <+\infty$ and $\Omega$ is an open set of $\mathbb{R}^3$, the space $E'$ is identified to $L^q(\Omega)$, $q=p/(p-1)$.
\item For $T>0$ and $E=L^1((0,T), L^2(\Omega))$, the space $L^\infty((0,T), L^2(\Omega))$ is identified with $E'$.
\end{itemize}

In the appendix, we give some useful technical lemmas.

\section{Analysis of the time semi-discrete incremental projection scheme} \label{sec:semidisc}

We consider a partition of the time interval $[0,T]$, which we suppose uniform to alleviate the notations, so that the assumptions read:

\begin{equation}\label{hyp:timedisc}
  N\ge 1, \qquad \deltat_{\! N} = \frac T N, \qquad t_N^n  = n\,\deltat_{\! N} \mbox{ for }  n \in \llbracket 0,N\rrbracket.\end{equation}
 

\subsection{The time semi-discrete scheme} 
Under the assumptions \eqref{hyp:timedisc}, the usual first order time semi-discrete incremental projection scheme (see \cite{Shen1992OnEE}) reads:  
\begin{subequations}\label{eq:semidiscrete}  
\begin{align} 
 & \mbox{\emph{Initialization:}} \nonumber \\
 &\hspace{2ex} \mbox{Let}~ \bfu_N^0 = \bfu_0 \in\bfE(\Omega) ~\text{and}~ p_N^0 =0. \label{eq:semidisc-init}\\[2ex]
 &\mbox{Solve for } 0 \le n \le N-1: \nonumber \\
 & \hspace{2ex}\mbox{\emph{Prediction step:}}   \nonumber \\
  & \label{eq:semidiscretepre}  \hspace{2ex}
\frac{1}{\deltat_{\! N}} ( \tilde{\bfu}_N^{n+1} - \bfu_N^\nalgo) +\dive(\tilde{\bfu}_N^{n+1}  \otimes \bfu_N^\nalgo)+ \nabla p_N^\nalgo  - \Delta \tilde{\bfu}_N^{n+1} = \bff_{\! N}^{n+1} \mbox{ in }\Omega, 
\\[1ex] \label{eq:semidiscretebound}
 & \hspace{1ex}
\tilde \bfu_N^{\nalgo+1} = 0\mbox{ on }\partial \Omega.
\\[2ex] 
  & \mbox{\emph{Correction step:}} \nonumber \\ 
  \label{eq:semidiscretecor}  
&\hspace{2ex}
%
\frac{1}{\deltat_{\! N}} (\bfu_N^{\nalgo+1} - \tilde{\bfu}_N^{n+1}) + \nabla (p_N^{n+1} - p_N^\nalgo) = 0 \mbox{ in }\Omega, 
\\[1ex] \label{eq:semidiscretediv} 
&  \hspace{2ex}
\dive \bfu_N^{\nalgo+1} = 0 \mbox{ in }\Omega \mbox{ and } \bfu_N^{\nalgo+1} \cdot \bfn = 0 \mbox{ on }\partial \Omega,
\\[1ex] \label{eq:semidiscreteintp} 
&  \hspace{2ex}
\int_\Omega p_N^{\nalgo+1} \dx = 0,
\end{align}
\end{subequations}
where $\bfn$ stands for the outward normal unit vector to the boundary $\partial \Omega$ and $\bff_{\! N}^{n+1} \in (L^2(\Omega))^3$ is defined by  
\[
  \bff_{\! N}^{n+1}(\bfx)=\dfrac 1 {\deltat_{\! N}}\int_{t^n}^{t^{n+1}} \bff(t,\bfx)  \dt,\ \text{for a.e.}~\bfx \in \Omega.
\]

Let us briefly account for the existence of a solution at each step of this algorithm.

\medskip

\noindent \textbf{Prediction step} --
A weak form of Eqns. \eqref{eq:semidiscretepre}-\eqref{eq:semidiscretebound} reads 
\begin{align}
&\mbox{Find } \tilde \bfu_N^{\nalgo+1} \in H_0^1(\Omega)^3 \mbox{ such that for any } \bfvarphi \in C_c^1(\Omega)^3, \nonumber\\
&\frac{1}{\deltat_{\! N}} \int_\Omega \tilde \bfu_N^{\nalgo+1} \cdot \bfvarphi \dx - \int_\Omega \tilde \bfu_N^{\nalgo+1} \otimes \bfu_N^\nalgo : \gradi \bfvarphi \dx + \int_\Omega \gradi \tilde \bfu_N^{\nalgo+1} : \gradi \bfvarphi \dx \label{eq:semi-weakpred}\\
&\qquad \qquad = \frac{1}{\deltat_{\! N}} \int_\Omega  \bfu_N^\nalgo \cdot \bfvarphi \dx + \int_\Omega p_N^\nalgo \dive \bfvarphi \dx + \int_\Omega \bff_{\! N}^{n+1} \cdot \bfvarphi \dx.\nonumber
\end{align}
The existence of the predicted velocity is then a consequence of Lemma \ref{lem:existpre}.

\medskip
\noindent \textit{Correction step} -- 
A weak form of Eqns. \eqref{eq:semidiscretecor}-\eqref{eq:semidiscretediv} reads
\begin{subequations}\label{eq:semi-weakcor}
\begin{align}
 &\mbox{Find } p_N^{n+1} \in H^1(\Omega) \mbox{ such that } \psi_N^{n+1} = p_N^{n+1} - p_N^n \in H^1(\Omega) \mbox{ satisfies :}   \\  
 &\int_\Omega \nabla \psi_N^{n+1} \cdot \nabla \varphi \dx = \frac 1 {\deltat_{\! N}} \int_\Omega \tilde \bfu_N^{n+1}\cdot \nabla \varphi \dx,~\text{for any}~ \varphi \in H^1(\Omega),  \\~\nonumber\\
 &\mbox{Set }\bfu_N^{\nalgo+1} =  \tilde \bfu_N^{\nalgo+1} -   {\deltat_{\! N}} \gradi  \psi_N^{n+1}.
\end{align}
\end{subequations}
If $\bfu_N^{\nalgo+1}$ satisfies \eqref{eq:semi-weakcor}, then $\int_\Omega  \bfu_N^{\nalgo+1}\cdot \gradi \varphi \dx = 0$ for any $\varphi  \in H^1(\Omega)$, so that $ \bfu_N^{n+1}$ belongs to  the space $\bfV(\Omega)$ of ``$L^2$-divergence-free functions'' defined by
\begin{equation}\label{def:V}
\bfV(\Omega)= \{ \bfu \in L^2(\Omega)^3~\text{such that}~\int_\Omega \bfu \cdot \gradi \xi \dx=0~\text{for any}~\xi \in H^1(\Omega)\}.
\end{equation}
The existence of $(\bfu_N^{\nalgo+1},p_N^{\nalgo+1}) \in \bfV(\Omega) \times H^1(\Omega) $ satisfying \eqref{eq:semi-weakcor}  is a consequence of the decomposition result of Lemma \ref{lem:decomp} given in the appendix.
Indeed, this correction step is the decomposition stated in Lemma \ref{lem:decomp} applied to the predicted velocity $\tilde \bfu^{\nalgo+1}$.
Note that $p_N^{\nalgo+1}$ is  uniquely defined thanks to \eqref{eq:semidiscreteintp}.

\begin{remark}
Summing \eqref{eq:semidiscretepre} at step $n$ and \eqref{eq:semidiscretecor} at step $n-1$, we obtain for $\nalgo \in \llbracket 1,N-1 \rrbracket$ 
\begin{equation}\label{eq:discretepre}
\frac{1}{\deltat_{\! N}} (\tilde \bfu_N^{n+1} - \tilde \bfu_N^n) + \dive (\tilde \bfu_N^{n+1} \otimes \bfu_N^n) + \gradi (2p_N^n -p_N^{n-1}) - \Delta \tilde \bfu_N^{n+1} =  \bff_{\! N}^{n+1}.
\end{equation}
\end{remark}

We may thus state the following existence result and define the approximate solutions obtained by the projection scheme \eqref{eq:semidiscrete}.
\begin{definition}[Approximate solutions, semi-discrete case] \label{def:exis-semid}
 Under the assumptions \eqref{hyp:T-Omega},\eqref{hyp:f-u0} and \eqref{hyp:timedisc},
 there exists $(\tilde \bfu_N^{\nalgo},\bfu_N^{\nalgo},$ $p_N^{\nalgo})_{\nalgo \in \llbracket 1,N \rrbracket} \subset$  $H_0^1(\Omega)^3\times \bfV(\Omega)$ $ \times H^1(\Omega)$ satisfying \eqref{eq:semidiscrete}. 
 We then define the functions  $\bfu_\nstep : (0,T) \to \bfV(\Omega)$ and $\tilde \bfu_\nstep: (0,T) \to H_0^1(\Omega)^3$ by
\begin{equation}\label{eqdef:fullfunctions-semid}
\bfu_\nstep (t) = \sum_{n=0}^{\nstep-1} \ \mathds 1_{(t^\nalgo_\nstep,t^{\nalgo+1}_\nstep]}(t)\bfu^{\nalgo}_\nstep,\quad \quad
\tilde \bfu_\nstep (t) = \sum_{n=0}^{N-1} \  \mathds 1_{(t^\nalgo_\nstep,t^{n+1}_\nstep]}(t) \tilde \bfu^{\nalgo+1}_\nstep,
\end{equation}
where  $(\tilde \bfu^{\nalgo}_\nstep)_{n \in \llbracket 1,N\rrbracket} $ and $(\bfu^{\nalgo}_\nstep)_{n \in \llbracket 1,N\rrbracket}$ are a solution to \eqref{eq:semidiscrete}, where $\mathds 1_A$ denotes the indicator function of a given set $A$.
 \end{definition}
 
\begin{remark}[On the boundary conditions] \label{rem:bc1}
The original homogeneous Dirichlet boundary conditions \eqref{boundary} of the strong formulation \eqref{pb:cont} is imposed on the weak solution through the functional space $H^1_0(\Omega)^3$. 
Note that this condition is only imposed on the predicted velocity in the algorithm  \eqref{eq:semidiscrete}. 
Indeed, the corrected velocity does not satisfy the full Dirichlet condition \eqref{boundary} but only the no slip condition imposed by \eqref{eq:semidiscretediv}.
The compactness of the sequence of predicted velocities  $\tilde \bfu$ together with the convergence of   $\bfu - \tilde \bfu$ towards zero in $L^2$ as the time step tends to zero will be the mean to prove that the Dirichlet boundary condition is finally satisfied on the limit of the numerical approximations.
Note also that there is no need for a boundary condition on the pressure in the correction step. 
In fact, it can be inferred from the correction step \eqref{eq:semi-weakcor} that the incremental pressure $\psi^{n+1} =  p^{n+1} - p^\nalgo$ satisfies a Poisson equation on $\Omega$ with a Neumann boundary condition on the boundary, but this is a redundant information that does not need to be implemented. 
We refer to \cite{rempfer-2006} for an interesting discussion on these boundary conditions. 
\end{remark}

\begin{remark}[On the initial condition]\label{rem:semid-init}
 In fact, the existence of a solution (see Lemma  \ref{lem:existpre}) only requires the initial velocity $\bfu^0_N$ to be in $\bfV(\Omega)$, so that we could relax the assumption on the initial condition $\bfu_0 \in \bfE(\Omega)$ to $\bfu_0 \in L^2(\Omega)^3$ and take $\bfu^0= \mathcal{P}_{\bfV(\Omega)} \bfu_0$ as the orthogonal projection of $\bfu_0$ onto the closed subspace $\bfV(\Omega)$ of  $L^2(\Omega)^3$, also known as the Leray projection. 
In this case, $\bfu^0_N$ can be computed as $\bfu^0  = \bfu_0 - \gradi \psi $ where $\psi \in H^1(\Omega)$ is a solution  (unique, up to a constant) of the following problem (see Lemma \ref{lem:decomp})
\begin{subequations}
\begin{align*}
 &\psi \in H^1(\Omega), \\  
 &\int_\Omega \gradi \psi \cdot \gradi \varphi \dx = \int_\Omega  \bfu_0 \cdot \gradi \varphi \dx,~\text{for any}~ \varphi \in H^1(\Omega).
\end{align*}
\end{subequations}

\end{remark}

\begin{theorem}[Convergence of the semi-discrete in time projection algorithm] \label{theo:conv-semid}
 Under the assumptions \eqref{hyp:T-Omega} and \eqref{hyp:f-u0}, consider for $N \ge 1$, the time discretization defined by \eqref{hyp:timedisc}, and the approximate solutions  $\bfu_N$ and $\tilde \bfu_N$ of the projection algorithm \eqref{eq:semidiscrete} as given in Definition \ref{def:exis-semid}.
 Then there exists $\bar \bfu \in L^2(0,T;\bfE(\Omega))$ $\cap$ $L^\infty(0,T;L^2(\Omega)^3)$ such that up to a subsequence, 
\begin{itemize}
  \item  the sequence  $(\tilde \bfu_N)_{N \ge 1}$  converges to $\bar \bfu$ in $L^2(0,T;L^2(\Omega)^3)$ and  weakly in $L^2(0,T;H_0^1(\Omega)^3)$,
  \item the sequence  $( \bfu_N)_{N \ge 1}$  converges to $\bar \bfu$  in $L^2(0,T;L^2(\Omega)^3)$ and $\star$-weakly in $L^\infty(0,T;L^2(\Omega)^3)$.
\end{itemize}
 Moreover the function $\bar \bfu$ is a weak solution to \eqref{pb:cont} in the sense of Definition \ref{def:weaksol}.
 \end{theorem}
\begin{proof}
 Here are the main steps of the proof; each step is detailed in one of the following paragraphs.
 \begin{itemize}
  \item \emph{Step 1: first estimates and weak convergence (detailed in section \ref{subsec:first-estim}).} 
    By Lemma \ref{lem:estsemidis} below, we get that there exists $\ctel{cste:est}\in \xR_+$, depending only on $|\Omega|$, $\Vert\bfu_0\Vert_{L^{2}(\Omega)^{3}}$ and $\Vert  \bff\Vert_{L^{2}(\Omega)^{3}}$, such that the sequences   $(\tilde \bfu_N)_{N \ge 1}$ and $(\bfu_N)_{N \ge 1}$ defined by \eqref{eqdef:fullfunctions-semid} satisfy
\begin{align} 
& \sup_{\nstep \ge 1}\|\tilde \bfu_\nstep\|_{L^2(0,T;H^1_0(\Omega)^3)} \leq \cter{cste:est} \quad \mbox{and } \quad \sup_{\nstep \ge 1}\|\bfu_\nstep\|_{L^{\infty}(0,T;L^{2}(\Omega)^{3})} \leq \cter{cste:est},
\label{eq:tildeuu_est} \\
& \sup_{\nstep \ge 1} \| \bfu_\nstep - \tilde \bfu_\nstep \|_{L^2(0,T;L^2(\Omega)^3)}
\le \cter{cste:est} \deltat_{\! N}. 
\label{diffutu}
\end{align} 
 
 Owing to \eqref{eq:tildeuu_est}, there exist some subsequences,  still denoted $(\bfu_\nstep)_{\nstep \ge 1}$ and $(\tilde \bfu_\nstep)_{\nstep \ge 1}$, that converge respectively  $\star$-weakly in $L^\infty(0,T;L^2(\Omega)^3)$ and weakly in $L^2(0,T;H_0^1(\Omega)^3)$.  
  Thanks to \eqref{diffutu}, the subsequences $(\bfu_\nstep)_{\nstep \ge 1}$ and $(\tilde \bfu_\nstep)_{\nstep \ge 1}$  converge to the same limit $\bar \bfu$ weakly in $L^2(0,T;L^2(\Omega)^3)$.
It follows that $\bar \bfu \in L^\infty(0,T;L^2(\Omega)^3) \cap L^2(0,T;H_0^1(\Omega)^3)$, and passing to the limit in the mass balance \eqref{eq:semidiscreteintp} then yields that $\bar \bfu \in L^\infty(0,T;L^2(\Omega)^3) \cap L^2(0,T;\bfE(\Omega))$. 

 \end{itemize}
  There remains to show that $\bar \bfu$ is a weak solution in the sense of Definition \ref{def:weaksol} and in particular that $\bar \bfu $ satisfies ($\ref{weaksol}$).
  Unfortunately, the weak convergence is not sufficient to pass to the limit in the scheme, because of the nonlinear convection term.
  Hence we first need to get some compactness on one of the subsequences (since, by \eqref{diffutu}, their difference tends to 0 in the $L^2$ norm).

\begin{itemize}  

  \item \emph{Step 2: compactness and convergence in $L^2$ (detailed in section \ref{subsec:compactness})} 
  This is the tricky part of the proof. 
  Since the sequence $(\tilde \bfu_\nstep)_{\nstep \ge 1}$ converges weakly in $L^2(0,T;H_0^1(\Omega)^3)$,  some estimate on the discrete time derivative would be sufficient to obtain the convergence in $L^2(0,T;H_0^1(\Omega)^3)$ by a Kolmogorov-like theorem. 
  A difficulty to obtain this estimate arises from the presence of the pressure gradient in Equation \eqref{eq:semidiscretepre}, which needs to be ``killed'' by multiplying this latter equation by a divergence-free function. 
  This function $\bfvarphi$ should also be regular enough so that the nonlinear divergence term makes sense: hence we choose $\bfvarphi \in L^2(0,T;W_0^{1,3}(\Omega)^3)$ such that $\dive \bfvarphi = 0$, and define the following semi-norm on $(L^2(\Omega))^3$:  
\begin{subequations}
 \label{etoile-un}
 \begin{align}
  &|\bfw|_{\ast,1} = \sup\{ \int_\Omega \bfw \cdot \bfv \dx,~ \bfv \in \bfW(\Omega),~ \Vert \bfv \Vert_{W_0^{1,3}(\Omega)^3} =1\},\label{etoile-un-seminorm}
  \\
  & \mbox{with }\bfW(\Omega) = \{\bfvarphi \in W_0^{1,3}(\Omega)^3~:~  \int_\Omega \bfvarphi \cdot \gradi \xi \dx = 0,   \forall  \xi \in H^1(\Omega)\},\label{etoile-un-space} 
\end{align}
\end{subequations}
 Estimates on the $L^2(|\cdot|_{\ast,1})$ semi-norm of the time translates of the predicted velocity $\widetilde \bfu_N$ are then obtained from the semi-discrete momentum equation \eqref{eq:semidiscretepre}: see Lemma \ref{lem:transsemidis}.
 Note that this is only an intermediate result; indeed, in order to gain compactness, we need an estimate on the time translates of the predicted velocity in the $L^2(L^2)$ norm. 
 The idea is then to first introduce the following semi-norm on $(L^2(\Omega))^3$.
\begin{equation}
\label{etoile-zero}
  |\bfv|_{\ast,0} = \sup\{ \int_\Omega \bfv \cdot \bfvarphi \dx,  ~\bfvarphi \in \bfV(\Omega),~\| \bfvarphi \|_{L^2(\Omega)^3}=1 \},
\end{equation}
where $\bfV(\Omega)$ is the space of $L^2$ divergence-free functions defined by \eqref{def:V}.
Note that  
\begin{equation}\label{ast0-PV}
|\bfw|_{\ast,0} = \| \mathcal{P}_{\bfV(\Omega)} \bfw \|_{L^2(\Omega)^3}~\text{for any}~ \bfw \in L^2(\Omega)^3,
\end{equation}
where  $\mathcal{P}_{\bfV(\Omega)}$ is the orthogonal projection operator onto $\bfV(\Omega)$.
 Then, thanks to a Lions-like lemma (Lemma \ref{lem:lions} below), we get that for any $\varepsilon >0$, there exists $C_\varepsilon \in \xR_+$ such that 
\begin{equation}\label{ineq:lions}
     |\bfw|_{\ast,0} \le \varepsilon \Vert \bfw\Vert_{H_0^1(\Omega)^3} + C_\varepsilon |\bfw|_{\ast,1}, \; \forall \bfw \in H^1_0(\Omega)^3.  
\end{equation}
By \eqref{eq:tildeuu_est}, we have an $L^2(0,T;H^1_0(\Omega)^3)$ bound on the predicted velocities; we have also seen that the time translates of $\widetilde \bfu_N$ for the $L^2(|\cdot|_{\ast,1})$ semi-norm tend to 0 as $N \to +\infty$ (Lemma \ref{lem:transsemidis} below).  
Therefore, by \eqref{ineq:lions}, the time translates of $\widetilde \bfu_N$ for the $L^2(|\cdot|_{\ast,0})$ semi-norm also tend to 0 as $N \to +\infty$.
In order to show that the $L^2(L^2)$ norm of the time translates of $\tilde \bfu_N$ tend to 0, we remark that if $\bfv \in \bfV(\Omega)$, then $|\bfv|_{\ast,0} = ||\bfv||_{L^2(\Omega)}$ and conclude thanks to \eqref{diffutu}, see Lemma \ref{lem:trans-utildesemidis}).

\smallskip
  
\item \emph{Step 3: convergence towards the weak solution (detailed in section \ref{subsec:passlim})} 
  Owing to a Kolmogorov-type theorem (see e.g. \cite[Corollary 4.41]{gh-edp}), the estimates of steps 1 and 2 yield that there exist subsequences, still denoted $(\bfu_\nstep)_{\nstep \ge 1}$ and $(\tilde \bfu_\nstep)_{\nstep \ge 1}$, that converge to  $\bar \bfu$ in $L^2(0,T;L^2(\Omega)^3)$.
  
 In section \ref{subsec:passlim}, we pass to the limit in the scheme to obtain that $\bar \bfu $ satisfies ($\ref{weaksol}$); therefore $\bar \bfu$ is a weak solution to \eqref{pb:cont} in the sense of Definition \ref{def:weaksol}.
  \end{itemize}
 \end{proof} 

\begin{remark}[Uniqueness and convergence of the whole sequence] \label{rem-uniq-instat}
	In the case where uniqueness of the solution is known, then the whole sequence converges ; this is for instance the case in the 2D case \cite{leray1934}, see e.g.  \cite[Chapter 5, Section 1.3]{BoyerFabrie-book} for more on this subject.
\end{remark}

\subsection{Proof of step 1: energy estimates and weak convergence} \label{subsec:first-estim}

\begin{lemma}[Energy estimates]\label{lem:estsemidis}  
Under the assumptions \eqref{hyp:T-Omega}, \eqref{hyp:f-u0} and \eqref{hyp:timedisc}, the functions
 $\bfu_\nstep$ and $\tilde \bfu_\nstep$ defined by \eqref{eqdef:fullfunctions-semid} satisfy \eqref{eq:tildeuu_est}  and \eqref{diffutu}, with $\cter{cste:est}$ depending only on $|\Omega|$, $\Vert\bfu_0\Vert_{L^{2}(\Omega)^{3}}$ and $\Vert  \bff\Vert_{L^{2}(\Omega)^{3}}$.
\end{lemma}

\begin{proof}

Noting that $\tilde \bfu_\nstep$ satisfies \eqref{eq:semi-weakpred} and using Lemma \ref{lem:existpre} with $\alpha = \frac{1}{\deltat_{\! N}}$, we have for $ n \in \llbracket 0,N-1 \rrbracket$
\begin{multline}
 \label{eq:estim_semid_pred}
\frac 1 {2\deltat_{\! N}} \Vert \tilde \bfu_\nstep^{\nalgo+1}\Vert_{L^2(\Omega)^3}^2 -\frac 1 {2\deltat_{\! N}}  \Vert \bfu_\nstep^\nalgo\Vert_{L^2(\Omega)^3}^2  +  \frac 1 {2\deltat_{\! N}} \Vert \tilde \bfu_\nstep^{\nalgo+1}- \bfu_\nstep^\nalgo\Vert_{L^2(\Omega)^3}^2   \\ + \int_\Omega \nabla p_\nstep^\nalgo \cdot \tilde \bfu_\nstep^{\nalgo+1} \dx + \Vert   \tilde \bfu_\nstep^{\nalgo+1}\Vert_{H^1_0(\Omega)^3}^2  \le \int_\Omega \bff_{\! N}^{n+1} \cdot \tilde \bfu_\nstep^{\nalgo+1} \dx. 
\end{multline}

Squaring the relation \eqref{eq:semidiscretecor}, integrating over $\Omega$,  multiplying by $\frac {\deltat_{\! N}}{ 2}$ and owing to 
$\bfu_N^{n+1} \in \bfV(\Omega)$, we get that for $ n \in \llbracket 0,N-1 \rrbracket$
\begin{multline*}
  \frac 1 {2 \deltat_{\! N}} \Vert \bfu_\nstep^{\nalgo+1}\Vert_{L^2(\Omega)^3}^2 + \frac{\deltat_{\! N}}{2} \|\nabla p_\nstep^{n+1}\|_{L^2(\Omega)^3}^2 \dx = \frac 1 {2 \deltat_{\! N}}   \|\tilde\bfu_\nstep^{\nalgo+1}\|_{L^2(\Omega)^3}^2 \\ + \frac{\deltat_{\! N}}{2}  \|\nabla p_\nstep^{n}\|_{L^2(\Omega)^3}^2 - \int_\Omega  \tilde\bfu_\nstep^{\nalgo+1} \cdot \nabla p_\nstep^\nalgo \dx.
\end{multline*}
Summing this latter relation with \eqref{eq:estim_semid_pred} yields for $n \in \llbracket 0,N-1 \rrbracket$
\begin{multline*}
  \frac 1 {2 \deltat_{\! N}} \left(\Vert\bfu_\nstep^{\nalgo+1}\Vert_{L^2(\Omega)^3}^2 - \Vert\bfu_\nstep^{\nalgo}\Vert_{L^2(\Omega)^3}^2\right)+ \frac{\deltat_{\! N}}{2} \left(\Vert\nabla p_N^{n+1}\Vert_{L^2(\Omega)^3}^2 - \Vert\nabla p_\nstep^{n}\Vert_{L^2(\Omega)^3}^2\right) \\ +   \frac{1}{2 \deltat_{\! N}} \Vert\tilde\bfu_\nstep^{\nalgo+1} - \bfu_\nstep^{\nalgo}\Vert_{L^2(\Omega)^3}^2 + \Vert  \tilde\bfu_\nstep^{\nalgo+1}\Vert_{H^1_0(\Omega)^3}^2  \le \int_\Omega \bff_{\! N}^{n+1} \cdot \tilde \bfu_N^{\nalgo+1} \dx.
\end{multline*}
We then get Relations \eqref{eq:tildeuu_est} by summing over the time steps,  using the Cauchy-Schwarz and Poincar\'e inequalities.
\end{proof}

\subsection{Proof of step 2: compactness and $L^2$ congergence} \label{subsec:compactness}

Following Step 2 of the sketch of proof of Theorem \ref{theo:conv-semid}, we start by the following lemma.
\begin{lemma}[A first estimate on the time translates]\label{lem:transsemidis}
Under the assumptions of Theorem \ref{theo:conv-semid}, there exists $\ctel{cste:esttrans1}$ only depending on $|\Omega|$, $\Vert\bfu_0\Vert_{(L^2(\Omega))^3}$ and $\Vert\bff\Vert_{(L^2(\Omega))^3}$such that for any $N \ge 1$ and for any $\tau \in (0,T)$, 
\begin{equation*}
\int_0^{T-\tau} |\tilde  \bfu_\nstep(t+\tau) -  \tilde \bfu_\nstep(t) |_{\ast,1}^2 \dt \le \cter{cste:esttrans1} \tau (\tau +\deltat_{\! N}), 
\end{equation*}
where $|\cdot|_{\ast,1}$ is the semi-norm defined by \eqref{etoile-un}.
\end{lemma}

\begin{proof}
Let $ N \ge 2$ and $\tau \in (0,T)$ (for $N=1$ the quantity we have to estimate is zero).  
Let $(\chi^n_{N,\tau})_{n \in \llbracket 1,N-1 \rrbracket}$  be the family of measurable functions defined for $n \in \llbracket 1,N-1 \rrbracket$ and $t \in \xR$ by $ \chi_{N,\tau}^n(t) =  \mathds 1_{ (t_N^n-\tau,t^n_N]}(t) $, then 
\begin{equation}
 \label{eqdef:chi-N-tau}
 \tilde \bfu_N(t+\tau)  -\tilde \bfu_N(t)  = \sum_{n=1}^{N-1} \chi_{N,\tau}^n(t) (\tilde \bfu_N^{n+1} - \tilde \bfu_N^n), \forall t \in (0,T-\tau).
\end{equation}

Hence, owing to \eqref{eq:discretepre},  
\begin{multline*}
\tilde \bfu_N(t+\tau) - \tilde \bfu_N(t) = \deltat_{\! N} \!\! \sum_{n=1}^{N-1} \!\!   \chi_{N,\tau}^n(t)\Delta \tilde \bfu_N^{n+1}  - \deltat_{\! N} \sum_{n=1}^{N-1} \!\!   \chi_{N,\tau}^n(t) \dive(\tilde  \bfu_N^{n+1}\otimes \bfu_N^{n}) \\
 - \deltat_{\! N} \sum_{n=1}^{N-1} \!\!   \chi_{N,\tau}^n(t) \gradi (2 p_N^{n} - p_N^{n-1}) + \deltat_{\! N} \sum_{n=1}^{N-1} \!\!   \chi_{N,\tau}^n(t) \bff_{\! N}^{n+1}.
\end{multline*}
Let $\bfvarphi  \in \bfW(\Omega)$ and $A(t)= \displaystyle\int_\Omega  \bigl( \tilde\bfu_N(t+\tau) - \tilde \bfu_N(t)\bigr) \cdot \bfvarphi \dx,$ so that
\begin{align*} 
& A(t)= A_{d}(t) +  A_{c}(t) + A_{p}(t) + A_{\bff}(t) \mbox{ with }\\
&
A_{d}(t) = -\sum_{n=1}^{N-1}\chi_{N,\tau}^n(t)\deltat_{\! N} \int_\Omega \gradi \tilde{\bfu}_N^{n+1} : \gradi \bfvarphi \dx,
\\ &
A_{c}(t) =\sum_{n=1}^{N-1}\chi_{N,\tau}^n(t)\deltat_{\! N} \int_\Omega \tilde{\bfu}_N^{n+1}  \otimes \bfu_N^{n} : \gradi \bfvarphi \dx,
\\ &
A_{p}(t) = \sum_{n=1}^{N-1}    \chi_{N,\tau}^n(t) \deltat_{\! N} \int_\Omega  (2 p_N^{n} - p_N^{n-1}) \dive \bfvarphi \dx,
\\ &
A_{\bff}(t) =  \sum_{n=1}^{N-1}    \chi_{N,\tau}^n(t)\deltat_{\! N} \int_\Omega \bff_{\! N}^{n+1}  \cdot  \bfvarphi \dx.
\end{align*}
By the H\"older inequality,
\begin{equation}\label{eq:est_Adsemi}
A_{d}(t) \le 
|\Omega|^{1/6} \| \bfvarphi \|_{W^{1,3}_0(\Omega)^3} \sum_{n=1}^{N-1}   \chi_{N,\tau}^n(t) \deltat_{\! N} \| \tilde{\bfu}_N^{n+1}  \|_{H_0^1(\Omega)^3}.
\end{equation}
Since $H^1_0(\Omega) \subset L^6(\Omega)$, using H\"older's inequality with exponents 2, 6 and 3 ($\frac 1 2 + \frac 1 6 + \frac 1 3 = 1$), we get, thanks to the bounds \eqref{eq:tildeuu_est} on $\widetilde \bfu_\nstep$ and $\bfu_\nstep$,
\begin{multline}\label{eq:est_Acsemi}
A_{c}(t)  \le \sum_{n=1}^{N-1}   \chi_{N,\tau}^n(t)  \deltat_{\! N}  \|\bfu_N^{n} \|_{L^2(\Omega)^3} \|  \tilde{\bfu}_N^{n+1} \|_{L^6(\Omega)^3} \|  \bfvarphi \|_{W^{1,3}_0(\Omega)^3} \\
\le \cter{cste:est} C_{{\rm sob}}^{(2,6)} \| \bfvarphi \|_{W_0^{1,3}(\Omega)^3}\sum_{n=1}^{N-1}  \chi_{N,\tau}^n(t) \deltat_{\! N}\| \tilde{\bfu}_N^{n+1}  \|_{H_0^1(\Omega)^3},
\end{multline}
 where  $C_{{\rm sob}}^{(2,6)}\in \xR_+$, depending only on $|\Omega|$, is such that
\begin{equation*}
\| \bfv \|_{L^6(\Omega)^3} \le C_{{\rm sob}}^{(2,6)} \| \bfv \|_{H_0^1(\Omega)^3},\mbox{ for any } \bfv \in H^1_0(\Omega)^3.
\end{equation*}
Since $\dive \bfvarphi =0$, clearly $A_p(t) =0$. 
Next, we note that
\begin{equation}\label{eq:est_Afsemi}
A_{\bff}(t) \le C_{{\rm sob}}^{(3,3)} |\Omega|^{1/6}\| \bfvarphi \|_{W_0^{1,3}(\Omega)^3} \sum_{n=1}^{N-1}    \chi_{N,\tau}^n(t)\deltat_{\! N} \|\bff_{\! N}^{n+1}\|_{L^2(\Omega)^3},
\end{equation}
 where $C_{{\rm sob}}^{(3,3)} \in \xR_+$, depending only on $|\Omega|$, is such that
\begin{equation*}
\| \bfvarphi \|_{L^3(\Omega)^3} \le C_{{\rm sob}}^{(3,3)} \| \bfvarphi \|_{W_0^{1,3}(\Omega)^3},\mbox{ for any } \bfvarphi \in W^{1,3}_0(\Omega)^3.
\end{equation*}
Summing Equations \eqref{eq:est_Adsemi}, \eqref{eq:est_Acsemi}, \eqref{eq:est_Afsemi}, we obtain
$$
A(t) \le C  \| \bfvarphi \|_{W_0^{1,3}(\Omega)^3}  \sum_{n=1}^{N-1}   \chi_{N,\tau}^n(t) \deltat_{\! N}( \| \tilde \bfu_N^{n+1} \|_{H_0^1(\Omega)^3} + \| \bff_{\! N}^{n+1} \|_{L^2(\Omega)^3})
$$
where $C = |\Omega|^{1/6} +C_{{\rm sob}}^{(3,3)}|\Omega|^{1/6}  +\cter{cste:est} C_{{\rm sob}}^{(2,6)}$.
This implies 
$$
|\tilde  \bfu_N(t+\tau) -  \tilde \bfu_N(t) |_{\ast,1} \le C  \sum_{n=1}^{N-1}  \chi_{N,\tau}^n(t) \deltat_{\! N} ( \| \tilde \bfu_N^{n+1} \|_{H_0^1(\Omega)^3} + \| \bff_{\! N}^{n+1} \|_{L^2(\Omega)^3}).
$$
Since $\sum_{n=1}^{N-1}  \chi_{N,\tau}^n(t) \deltat_{\! N} \le \tau +\deltat_{\! N} $ for any $t \in (0,T-\tau)$ we then obtain
$$
|\tilde  \bfu(t+\tau) -  \tilde \bfu(t) |_{\ast,1}^2 \le 2 C^2 (\tau +\deltat_{\! N})  \sum_{n=1}^{N-1} \chi_{N,\tau}^n(t)\deltat_{\! N} ( \| \tilde \bfu_N^{n+1} \|_{H_0^1(\Omega)^3}^2 + \| \bff_{\! N}^{n+1} \|_{L^2(\Omega)^3}^2).
$$
Noting that  $\int_0^{T-\tau} \chi_{N,\tau}^n(t) \dt \le \tau $ for any $n \in \llbracket 1,N-1 \rrbracket$ yields
\begin{multline*}
\int_0^{T-\tau} |\tilde  \bfu_N(t+\tau) -  \tilde \bfu_N(t) |_{\ast,1}^2 \dt  \\  \le 2 C^2 (\tau +\deltat_{\! N}) \sum_{n=1}^{N-1}  \deltat_{\! N} ( \| \tilde \bfu_N^{n+1} \|_{H_0^1(\Omega)^3}^2 + \| \bff_{\! N}^{n+1} \|_{L^2(\Omega)^3}^2) \int_0^{T-\tau}  \chi_{N,\tau}^n(t) \dt \\
\le   2 C^2 (\tau +\deltat_{\! N}) \tau ( \| \tilde \bfu_N \|_{L^2(0,T:H_0^1(\Omega)^3)}^2 + \| \bff \|_{L^2((0,T) \times \Omega)^3}^2) 
\le \cter{cste:esttrans1} \tau (\tau +\deltat_{\! N})
\end{multline*}
which gives the expected result.
\end{proof}

\begin{lemma}[Lions-like]\label{lem:lions}
Let $\Omega$ be an open bounded connected subset of $\xR^3$ with a Lipschitz boundary.
For any $\varepsilon >0$, there exists $C_\varepsilon >0$ such that \eqref{ineq:lions} holds for any $\bfw \in H^1_0(\Omega)^3$.
\end{lemma}
\begin{proof}
Let $\varepsilon >0$; let us show by contradiction that there exists $C_\varepsilon >0$ such that for any $\bfw \in H_0^1(\Omega)^3$
\begin{equation*}
    |\bfw|_{\ast,0}  \le \varepsilon \Vert \bfw \Vert_{H_0^1(\Omega)^3} + C_\varepsilon |\bfw|_{\ast,1}. 
    \label{eq:f}
\end{equation*}
Suppose that this is not so, then there exists $ \varepsilon >0$ and a sequence $(\bfw_n)_{n \ge 0}$ of functions of $H_0^1(\Omega)^3$, such that,thanks to \eqref{ast0-PV},
\begin{equation*}
  \|  \mathcal{P}_{\bfV(\Omega)} \bfw_n \|_{L^2(\Omega)^3} =   |\bfw_n|_{\ast,0} > \varepsilon \Vert \bfw_n \Vert_{H_0^1(\Omega)^3} + n |\bfw_n|_{\ast,1}. 
\end{equation*}
By a homogeneity argument, we may choose $\Vert \mathcal{P}_{\bfV(\Omega)} \bfw_n\Vert_{L^2(\Omega)^3} = 1$; it then follows from the latter inequality that the sequence $(\bfw_n)_{n \ge 0}$ is bounded in $H_0^1(\Omega)^3$ and that $|\bfw_n|_{\ast,1} \to 0$ as $\nti$.
This implies that as $\nti$,  up to a subsequence, $ (\bfw_n)_{n \ge 0} $ converges  in $L^2(\Omega)^3$ to $\bfw \in H_0^1(\Omega)^3$.
The continuity of the Leray projection $\mathcal{P}_{\bfV(\Omega)}$ implies that 
$ \mathcal{P}_{\bfV(\Omega)} \bfw_{n} \to \mathcal{P}_{\bfV(\Omega)}  \bfw$ in  $L^2(\Omega)^3$ and in particular $\|\mathcal{P}_{\bfV(\Omega)}  \bfw \|_{L^2(\Omega)^3}=1$.
By definition of $|\bfw_n|_{\ast,1}$ we have for any $\bfvarphi \in \bfW(\Omega)$
\[
   \int_\Omega \bfw_n \cdot\bfvarphi\dx \le |\bfw_n|_{\ast,1} \Vert \bfvarphi \Vert_{W_0^{1,3}(\Omega)^3}.
\]
We then obtain
\[
   \int_\Omega \mathcal{P}_{\bfV(\Omega)} \bfw_n \cdot\bfvarphi\dx =  \int_\Omega \bfw_n \cdot\bfvarphi\dx  \le |\bfw_n|_{\ast,1} \Vert \bfvarphi \Vert_{W_0^{1,3}(\Omega)^3}.
\]
Passing to the limit in this inequality yields that 
\[
  \int_\Omega \mathcal{P}_{\bfV(\Omega)}\bfw \cdot \bfvarphi \dx = 0,~\text{for any}~ \bfvarphi \in \bfW(\Omega).
\]
Owing to Lemma \ref{lem:carac-grad}, this in turn implies that there exists 
$\xi \in H^1(\Omega)$ such that $\mathcal{P}_{\bfV(\Omega)} \bfw = \gradi \xi$. 
Using the fact that $ \mathcal{P}_{\bfV(\Omega)} \bfw \in \bfV(\Omega)$ we have
$$
\| \mathcal{P}_{\bfV(\Omega)} \bfw \|_{L^2(\Omega)^3}^2 = \int_\Omega \mathcal{P}_{\bfV(\Omega)} \bfw \cdot \gradi \xi \dx = 0,
$$
which contradicts $\Vert \mathcal{P}_{\bfV(\Omega)} \bfw \Vert_{L^2(\Omega)^3} = 1$.
\end{proof}

\begin{lemma}[$L^2$ estimate on the time translates]\label{lem:trans-utildesemidis}
Under the assumptions Theorem \ref{theo:conv-semid}, the sequence $(\tilde \bfu_\nstep)_{N \ge 1}$ satisfies 
\begin{equation}
\label{eq:trans-utildesemidis}
 \int_0^{T-\tau}\!\!\!\!\!\!\!\!\Vert \tilde \bfu_\nstep(t+\tau) - \tilde \bfu_\nstep(t) \Vert_{L^2(\Omega)^3}^2 \dt \to 0 \mbox{ as } \tau \to 0, \mbox{ uniformly with respect to }N, 
\end{equation}
and is therefore relatively compact in $L^2(0,T;L^2(\Omega)^3)$.
\end{lemma}
\begin{proof}
 By the triangle inequality, 
 \begin{align*}
  &\int_0^{T-\tau} \Vert \tilde \bfu_\nstep(t+\tau) - \tilde \bfu_\nstep(t) \Vert_2^2 \dt \le
2 (A_\nstep(\tau) + B_\nstep(\tau)), \mbox{ with } \\
  & A_\nstep(\tau) =  \int_0^{T-\tau} \Vert (\tilde \bfu_\nstep -\bfu_\nstep)(t+\tau) - (\tilde \bfu_\nstep-\bfu_\nstep)(t) \Vert_2^2 \dt, \\
  & B_\nstep(\tau) =  \int_0^{T-\tau} \Vert   \bfu_\nstep(t+\tau) - \ \bfu_\nstep(t) \Vert_2^2 \dt. 
 \end{align*}
For any fixed $N\in \xN$,  $A_\nstep(\tau) \to 0$ as $\tau \to 0$, and thanks to \eqref{diffutu}, this convergence is uniform with respect to $N$.
Let us then show that $B_\nstep(\tau) \to 0$ as $\tau \to 0$ uniformly with respect to $N$.

Since $\bfu_\nstep(t) \in \bfV(\Omega)$ for any $t \in (0,T)$ we have for any $t \in (0,T-\tau)$
\begin{multline*}
\Vert \bfu_\nstep(t+\tau) -  \bfu_\nstep(t) \Vert_{L^2(\Omega)^3} =\!\!\!\sup_{\substack{\bfv \in  \bfV(\Omega)\\ \ \Vert \bfv\Vert_{L^2(\Omega)^3} = 1}}\!\!\!\int_\Omega \left(\bfu_\nstep(t+\tau) -  \bfu_\nstep(t)\right) \cdot \bfv \dx  \\
 \le \Vert  (\bfu_\nstep - \tilde \bfu_\nstep)(t+\tau) -  (\bfu_\nstep - \tilde \bfu_\nstep)(t) \Vert_{L^2(\Omega)^3} + \!\!\!\!\!\! \sup_{\substack{\bfv \in  \bfV(\Omega)\\ \ \Vert \bfv\Vert_{L^2(\Omega)^3} = 1}}\!\!\!\!\!\!\int_\Omega (\tilde \bfu_\nstep(t+\tau) -  \tilde \bfu_\nstep(t)) \cdot \bfv \dx,
\end{multline*}
so that 
\[
B_N(\tau)  \le 2 A_\nstep(\tau) +  2 \int_0^{T-\tau} |\tilde  \bfu_\nstep(t+\tau) -  \tilde \bfu_\nstep(t) |_{\ast,0}^2 \dt
\]
 Let $\varepsilon >0$;  thanks to Lemma \ref{lem:lions}, there exists $C_\varepsilon  >0$ such that for any $N \ge 1$ and for any $t \in (0,T-\tau)$
\begin{align*}
 &|\tilde  \bfu_\nstep(t+\tau) -  \tilde \bfu_\nstep(t) |_{\ast,0} \le \varepsilon \Vert \tilde  \bfu_\nstep(t+\tau) -  \tilde \bfu_\nstep(t) \Vert_{H_0^1(\Omega)^3} + C_\varepsilon |\tilde  \bfu_\nstep(t+\tau) -  \tilde \bfu_\nstep(t) |_{\ast,1}, 
\end{align*}
and in particular for any $N \ge 1$ and $\tau \in (0,T)$
\begin{multline*}
 \int_0^{T-\tau} |\tilde  \bfu_\nstep(t+\tau) -  \tilde \bfu_\nstep(t) |^2_{\ast,0} \dt \le 2 \varepsilon^2 \int_0^{T-\tau} \Vert \tilde  \bfu_\nstep(t+\tau) -  \tilde \bfu_\nstep(t) \Vert_{H^1_0(\Omega)^3}^2 \dt \\ + 2C_\varepsilon^2 \int_0^{T-\tau} |\tilde  \bfu_\nstep(t+\tau) -  \tilde \bfu_\nstep(t) |^2_{\ast,1} \dt.
\end{multline*}
Thus, owing to lemmas \ref{lem:estsemidis} and \ref{lem:transsemidis}, 
\begin{equation*}
\int_0^{T-\tau} |\tilde  \bfu_\nstep(t+\tau) -  \tilde \bfu_\nstep(t) |^2_{\ast,0} \dt 
\le 8 \cter{cste:est}^2 \varepsilon^2  + 2C_\varepsilon^2 \cter{cste:esttrans1} \tau (\tau+\deltat_{\! N}),
\end{equation*}
and therefore, for any $ N \ge 1$ and $\tau \in (0,T)$,
$$
B_\nstep(\tau) \le 2 A_\nstep(\tau) +  16 \cter{cste:est}^2 \varepsilon^2 + 4 C_\varepsilon^2 \cter{cste:esttrans1} \tau(\tau+\deltat_{\! N}).
$$

Now let $\zeta > 0$ be given, and let:
\begin{itemize}
\item  $\tau_0 > 0$ such that for any $ \tau \in (0,\tau_0)$,   $ 2 A_\nstep(\tau) \le \zeta $ for any $N \ge 1$;
\item $\varepsilon > 0$ such that $ 16 \cter{cste:est}^2 \varepsilon^2 \le \zeta$;
\item $\tilde \tau_0 > 0$ 
such that for any $ \tau \in (0,\tilde \tau_0) $ 
and $N \ge 1$, $4 C_{\varepsilon}^2 \cter{cste:esttrans1} \tau (\tau+\deltat_{\! N}) \le \zeta$.
\end{itemize}

We then obtain that $B_\nstep(\tau) \le 3 \zeta$ for any $ \tau \in (0, \min(\tau_0,\tilde \tau_0)) $ and 
 $N \ge 1$ which implies that $ B_\nstep(\tau)  \to 0$  as $\tau \to 0$, uniformly with respect to $N$. 
 The proof of \eqref{eq:trans-utildesemidis} is thus complete.
 The relative compactness of the sequences $\widetilde \bfu_N$ and $\bfu_N$ follows by a Kolmogorov-like theorem (see e.g. \cite[Corollary 4.40]{gh-edp}) and \eqref{diffutu}.
\end{proof}

\subsection{Proof of step 3: convergence to a weak solution}\label{subsec:passlim}
By Lemma \ref{lem:trans-utildesemidis}, up to a subsequence, the sequence of predicted velocities $(\tilde \bfu_N)_{N \ge 1}$ converges to some limit $\bar \bfu \in (L^2(0,T;L^2(\Omega)^3)$,  and owing to \eqref{diffutu}, so does the sequence $(\bfu_N)_{N \ge 1}$.
There remains to check that $\bar \bfu$ is a weak solution to \eqref{pb:cont} in the sense of Definition \ref{def:weaksol}.
This is a result that we call ``Lax-Wendroff consistency'', following the famous paper \cite{lax-60-sys} see e.g. \cite{eym-22-lax}: assuming that the approximate solutions converge boundedly to a limit, this limit is a weak solution to the continuous problem. 

\begin{lemma}[Lax-Wendroff consistency of the semi-discrete scheme]\label{lem:weaksolsemidis}~`

Let $(\tilde \bfu_N)_{N \ge 1}\subset L^2(0,T;H^1_0(\Omega)^3)$ and $(\bfu_N)_{N \ge 1}$ $ \subset L^\infty(0,T;L^2(\Omega)^3)$ be sequences of solutions to the semi-discrete scheme \eqref{eq:semidiscrete} for $N \in \xN$ (see Definition \ref{def:exis-semid}), and assume that $\bar \bfu \in L^2(0,T;H^1_0(\Omega)^3)$ is such that $\tilde \bfu_N \to \bar \bfu$ in $L^2(0,T;L^2(\Omega)^3)$ and weakly in $L^2(0,T;H^1_0(\Omega)^3)$ and $\tilde \bfu_N \to \bar \bfu$ $\star$-weakly in $L^\infty(0,T;L^2(\Omega)^3)$ as $N\to + \infty$. 
Then the function $\bar \bfu$ is a weak solution to \eqref{pb:cont} in the sense of Definition \ref{def:weaksol}.
 \end{lemma}
 \begin{proof}
Let $\bfvarphi \in\bigl \{\bfw\in C_c^\infty ( [0,T) \times \Omega)^3,\ \dive \bfvarphi=0~\mbox{in}  (0,T) \times \Omega \bigr\}$. 
Let $(\bfvarphi_N^n)_{n \in \llbracket 0,N \rrbracket}$ be the sequence of functions of $\mathbf{E}(\Omega)$ defined by
$
\bfvarphi_N^n(\bfx)  = \bfvarphi(t_N^n,\bfx),\ \text{for any}~\bfx \in \Omega,
$
and let $\bfvarphi_N : (0,T) \to \mathbf{E}(\Omega)$ and $\bff_{\! N} : (0,T) \to L^2(\Omega)^3$ be defined by
\begin{equation*}
 \bfvarphi_\nstep (t) = \sum_{n=0}^{\nstep-1} \ \Char_{(t_N^\nalgo,t_N^{n+1}]}(t) \bfvarphi_N^{\nalgo}, \quad \quad \bff_\nstep (t) = \sum_{n=0}^{\nstep-1} \ \Char_{(t_N^\nalgo,t_N^{n+1}]}(t) \bff^{n+1}_{N}.
\end{equation*}
The regularity of $\bff$ and $\bfvarphi$ implies that:
\[
 \left.\begin{array}{l}
        \Vert \bff_{\! N} - \bff \Vert_{L^2((0,T) \times \Omega)^3} \to 0,\\
        \Vert \bfvarphi_N - \bfvarphi \Vert_{L^\infty((0,T) \times \Omega)^3} \to 0,\\
        \Vert  \gradi \bfvarphi_N - \gradi \bfvarphi \Vert_{L^\infty((0,T) \times \Omega)^{3 \times 3}} \to 0,
       \end{array}
\right\} \mbox{ as } N \to +\infty.
\]
Multiplying \eqref{eq:discretepre} by $\deltat_{\! N} \bfvarphi_N^{n}$, integrating over $\Omega$ and summing over $n \in \llbracket 1,N-1 \rrbracket$ yields
\begin{multline}\label{eq:weaksolsemidisRvarphi}
\sum_{n=1}^{N-1} \int_\Omega (\tilde\bfu_N^{n+1} - \tilde \bfu_N^n) \cdot \bfvarphi_N^n \dx  - \int_{\deltat_{\! N}}^T \int_\Omega \tilde \bfu_N \otimes \bfu_N : \gradi \bfvarphi_N \dx \dt \\ + \int_{\deltat_{\! N}}^T \int_\Omega \gradi \tilde \bfu_N : \gradi \bfvarphi_N \dx \dt = \int_{\deltat_{\! N}}^T \int_\Omega \bff_{\! N} \cdot \bfvarphi_N \dx \dt. 
\end{multline}
Using the fact that $\bfvarphi_N^{N}=0$ in $\Omega$ the first term of the left hand side reads
\begin{equation*}
\sum_{n=1}^{N-1} \int_\Omega (\tilde\bfu_N^{n+1} - \tilde \bfu_N^n) \cdot \bfvarphi_N^n \dx =  -  \int_0^T \int_\Omega \tilde \bfu_N^n \cdot (\bfvarphi_N^n - \bfvarphi_N^{n-1}) \dx \dt - \int_\Omega \tilde \bfu_N^1 \cdot \bfvarphi_N^0 \dx.
\end{equation*}

By the triangle inequality, 
\[
 \int_\Omega \tilde \bfu_N^1 \cdot \bfvarphi_N^0 \dx = \int_\Omega \bfu_0 \cdot \bfvarphi(0,\cdot) \dx +  \int_\Omega (\tilde \bfu_N^1 - {\bfu_0}) \cdot \bfvarphi(0,\cdot) \dx.
\]
Since the sequence $(\tilde \bfu_N)_{N \ge 1}$ converges to $\bar \bfu$ in $L^2(0,T;L^2(\Omega)^3)$, we obtain
\begin{equation}\label{eq:convweakstep1}
\lim_ {\Nti}\sum_{n=1}^{N-1} \int_\Omega (\tilde\bfu_N^{n+1} - \tilde \bfu_N^n) \cdot \bfvarphi_N^n \dx  = - \int_0^T \int_\Omega \bar \bfu \cdot \partial_t \bfvarphi \dx - \int_\Omega \bfu_0 \cdot \bfvarphi(0,\cdot) \dx.
\end{equation}
The second term in the left hand-side reads
\begin{multline*}
 \int_{\deltat_{\! N}}^T \int_\Omega \tilde \bfu_N \otimes \bfu_N : \gradi \bfvarphi_N \dx \dt = \int_0^T \int_\Omega \tilde \bfu_N \otimes \bfu_N : \gradi \bfvarphi \dx \dt  \\
+\int_0^T \int_\Omega \tilde \bfu_N \otimes \bfu_N :( \gradi \bfvarphi_N - \gradi \bfvarphi) \dx \dt  - \int_0^{\deltat_{\! N}} \int_\Omega \tilde \bfu_N \otimes \bfu_N : \gradi \bfvarphi_N \dx \dt
\end{multline*}
The convergence of the sequence $(\tilde \bfu_N)_{N \ge 1}$ in $L^2(0,T;L^2(\Omega)^3)$, the weak convergence of the sequence $(\bfu_N)_{N \ge 1}$ in $L^2(0,T;L^2(\Omega)^3)$,  the convergence of the sequence $ (\gradi \bfvarphi_N)_{N \ge 1}$ in $L^\infty((0,T)\times \Omega)^{3 \times 3}$  implies
\begin{equation}\label{eq:convweakstep2}
\lim_{\Nti}  \int_{\deltat_{\! N}}^T \int_\Omega \tilde \bfu_N \otimes \bfu_N : \gradi \bfvarphi_N \dx \dt  = \int_0^T \int_\Omega \bar \bfu \otimes \bar \bfu : \gradi \bfvarphi \dx \dt.
\end{equation}
The third term in the left hand-side may be written
\begin{multline*}
\int_{\deltat_{\! N}}^{T} \int_\Omega \gradi \tilde \bfu_N : \gradi \bfvarphi_N \dx \dt =\int_0^T \int_\Omega \gradi \tilde \bfu_N : \gradi \bfvarphi \dx \dt  \\
+ \int_0^T \int_\Omega \gradi \tilde \bfu_N :( \gradi \bfvarphi_N - \gradi \bfvarphi) \dx \dt- \int_0^{\deltat_{\! N}} \int_\Omega \gradi \tilde \bfu_N  : \gradi \bfvarphi_N \dx \dt 
\end{multline*}
The weak convergence of the sequence $(\gradi \tilde \bfu_N)_{N \ge 1}$ in $L^2(0,T;L^2(\Omega)^3)$ and the convergence of the sequence $ ( \gradi \bfvarphi_N)_{N \ge 1}$ in $L^2(0,T;L^2(\Omega)^3)$ implies
\begin{equation}\label{eq:convweakstep3}
\lim_{\Nti}\int_{\deltat_{\! N}}^{T} \int_\Omega \gradi \tilde \bfu_N : \gradi \bfvarphi \dx \dt = \int_0^T \int_\Omega \gradi \bar \bfu : \gradi \bfvarphi \dx \dt.
\end{equation}
The right hand-side satisfies
\begin{multline*}
\int_{\deltat_{\! N}}^T \int_\Omega \bff_{\! N} \cdot \bfvarphi_N \dx \dt= \int_0^T \int_\Omega \bff_{\! N} \cdot \bfvarphi \dx \dt 
\int_0^T \int_\Omega \bff_{\! N} \cdot ( \bfvarphi_N - \bfvarphi) \dx \dt \\   -\int_0^{\deltat_{\! N}} \int_\Omega \bff_{\! N} \cdot \bfvarphi_N \dx \dt.
\end{multline*}
The convergence of the sequence $( \bff_{\! N})_{N \ge 1}$ in $L^2(0,T;L^2(\Omega)^3)$ and  the convergence of the sequence $ (  \bfvarphi_N)_{N \ge 1}$ in $L^2(0,T;L^2(\Omega)^3)$  implies
\begin{equation}\label{eq:convweakstep4}
\lim_{\Nti} \int_{\deltat_{\! N}}^T \int_\Omega \bff_{\! N} \cdot \bfvarphi_N \dx \dt = \int_0^T \int_\Omega \bff \cdot \bfvarphi \dx \dt.
\end{equation}

Using ($\ref{eq:convweakstep1}$)-($\ref{eq:convweakstep4}$) and passing to the limit in ($\ref{eq:weaksolsemidisRvarphi}$) gives the expected result.
\end{proof}

\section{Analysis of the fully discrete projection scheme} \label{sec:MAC-def}
Our purpose is now to adapt the proof of convergence of the semi-discrete case to the fully discrete case. 
We choose as an example of space discretization a staggered discretization on a (possibly non uniform) rectangular grid of $\xR^3$.
The resulting scheme, often referred to as a MAC scheme, was analysed in \cite{gal-18-convmac} for an time-implicit scheme. 
The idea here is to prove its convergence for the incremental projection scheme.
We consider the following assumptions on $\Omega$ and on the time-space discretization, indexed by $N$ (in the convergence analysis, the time and space steps will tend to 0 as $N$ tends to $+\infty$).
\begin{subequations}\label{hyp:timespacedisc}
\begin{align}
   &\nonumber
    T >0, \quad \Omega \mbox{ is an open rectangular subset of }  \xR^3, \mbox{ with boundary }\\
   & \label{hyp:T-Omega-rec}\mbox{faces that are orthogonal to one of the vectors of the canonical basis of} \\
   &\nonumber \xR^3, \mbox{denoted by }\{\bfe^{(i)}, i=1,2,3\}, \\[1ex]
    \label{hyp:time}
   & N\ge 1, \qquad \deltat_{\! N} = \dfrac T N, \qquad t_N^n  = n\,\deltat_{\! N} \mbox{ for }  n \in \llbracket 0,N\rrbracket.\\[1ex]
   \nonumber 
   &\mathcal{D}_N= (\mesh_N, \edges_N) \mbox{ is a MAC discretization in the sense of}  \\
    &\label{hyp:mesh}\mbox{\cite[Definition 2.1]{gal-18-convmac}},\mbox{ with } \mesh_N \mbox{(resp. } \edges_N) \mbox{ the set of cells }\mbox{(resp. faces)},\\ 
   &\nonumber h_N = \max_{K \in \mesh_N} \mathrm{diam} K \mbox{ is the space step. }
\end{align}
\end{subequations}

Note that at this step, we are only considering one time step $\deltat_{\! N} = \frac T N$ and one discretization mesh $\disc_N$, which is also  indexed by $N$. 
This might seem strange, but it is in view of  the convergence analysis for which a sequence $(\disc_N, \deltat_{\! N})_{N\ge 1}$ will be considered, with $h_N, \deltat_{\! N} \to 0$ as $N\to + \infty$.

The regularity of the mesh is measured by the following parameter:
\begin{equation}\label{eqdef:thetaN}
\theta_N =\max \Bigl\{ \frac{|\edge|}{|\edge'|},
\ \edge \in \edgesi,\ \edge' \in \edgesj,\ i, j \in \llbracket 1, d\rrbracket,\ i\not= j \Bigr\},
\end{equation}
with $|\cdot|$ the Lebesgue measure (this notation is used in the following for either the $\xR^3$ or the $\xR^{2}$ measure).

We refer to \cite{gal-18-convmac} for the precise definition ot the discrete spaces and operators.
The approximate pressure belongs to the set $L_N(\Omega)$ of functions that are piecewise constant on the so called primal cells $K$ of the (primal) mesh $\mesh_N$: $p = \sum_{K\in \mesh_N} p_K \mathds 1_K$. 
The $i$-th component of the approximate velocities  belongs to the set $\Hmeshi(\Omega)$  of functions that are piecewise constant on the dual cells $D_\edge \in \edgesi$, where $\edgesi$ denotes the set of faces of the mesh that are orthogonal to $\bfe_i$. 
Denoting by $\edges(K)$ the set of faces of a given cell $K \in \mesh_N$, and by $\edge=K|L$ an interface between two neighbouring cells $K$ and $L$, a dual cell $D_\edge \in \edges \cap \edges(K)$ is defined by
\[D_\sigma =
\begin{cases}
  [\bfx_K \bfx_L] \times \edge,~\text{for}~\edge=K|L \subset \Omega, 
  \\
  [\bfx_K  \bfx_{K,\partial\Omega}] \times \edge,~\text{for}~\sigma \subset \partial \Omega.
\end{cases}
\]
where $x_K$ denotes the mass center of $K$ and $\bfx_{K,\partial\Omega}$ the orthogonal projection of $\bfx_K$ on ${\partial\Omega}$
We thus define three dual meshes of $\Omega$.
 
\subsection{The fully discrete scheme} \label{sec:scheme}
The space discretization of the time-discrete scheme \eqref{eq:semidiscrete} reads:
\begin{subequations} \label{eq:scheme}
\begin{align} 
 & \mbox{\emph{Initialization:}} \nonumber \\
 &\hspace{2ex} 
  \bfu_N^0 = (u^0_{N,i})_{i=1,2,3} \mbox{ with } u^0_{N,i} = \sum_{\edge \in \edges_N^{(i)}} \frac 1 {|\edge|} \!\! \int_\edge\!\!\ u_{0,i}(s) \ds \ \mathds 1_{D_\edge}, i=1,2,3, \label{scheme:init}\\&\hspace{2ex} \vspace{-.5cm} p^0 =0. \nonumber
 \\[2ex]
 \nonumber & \mbox{\emph{Solve for }} 0\leq n \leq N-1,
    \\[1ex]  &  \hspace{1ex} 
 \nonumber  \mbox{\emph{Prediction step}}  
\\[1ex] 
 \label{scheme:pred_int} &  \hspace{2ex}
     \frac{1}{\deltat_{\! N}}( \widetilde \bfu_{N}^{n+1} - \bfu_N^{n} ) + \bfC_{\! N}(\tilde \bfu_N^{\nalgo+1}) \bfu_N^n -\Delta_N \tilde \bfu_N^{\nalgo+1}\!\!  + \gradi_{\! N} p_N^\nalgo  = \bff_{\! N}^{n+1}~\text{in}~\Omega.
\\[1ex] 
\label{scheme:pred_ext} &  \hspace{2ex}  
 (\tilde u_N^{n+1})_\sigma = 0,~ \forall \edge \in \edgesext.
 \\[1ex]    &  \hspace{1ex}
\nonumber  \mbox{\emph{Correction step}} : 
\\[1ex] \label{scheme:cor_int} & \hspace{2ex}
\frac{1}{\deltat_{\! N}}(\bfu_N^{n+1} - \tilde \bfu_N^{n+1}) + \gradi_{\! N} (p_N^{n+1}-  p_N^\nalgo) =0~\text{in}~\Omega,
\\[1ex] \label{schemediv} & \hspace{2ex}
\dive_{\! N}\bfu_N^{n+1} =0~\text{in}~\Omega,
\\[1ex] \label{scheme:cor_ext} & \hspace{2ex}
(u_N^{n+1})_\sigma = 0,~ \forall \edge \in \edgesext,
\\[1ex] \label{scheme:pmean0} & \hspace{2ex}
\sum_{K \in \mesh} |K|\ p_K^{n+1} =0.
\end{align}
\end{subequations}
In this algorithm, the terms $\bfC_{\! N}(\tilde \bfu_N^{\nalgo+1})\bfu_N^n$,  -$\Delta_N \tilde \bfu_N^{\nalgo+1}$,  $\nabla_{\! N} p_N^\nalgo$ and $\dive_{\! N}\bfu_N^{n+1}$ are the MAC discretization of the terms  $\dive(\tilde{\bfu}_N^{n+1}  \otimes \bfu_N^\nalgo)$,  $\Delta \tilde{\bfu}_N^{n+1}$, $\gradi p_N^\nalgo$ and $\dive \bfu_N^{n+1}$ in the algorithm \eqref{eq:semidiscrete} and are defined in \cite[Section 2]{gal-18-convmac}.
In \eqref{scheme:pred_ext}, the vector function $\bff_{\! N}^{n+1}$ is defined by its components $(f_{N,i}^{n+1}, i=1,2,3)$ where $f_{N,i}^{n+1}$ is the piecewise constant function from $\Omega\times(0,T)$ to $\xR^3$ defined by 
\[
  \bff_{\! N}^{n+1}(\bfx)=\dfrac 1 {|D_\edge|}\dfrac 1 {\deltat_{\! N}}\int_{D_\edge}\int_{t^n}^{t^{n+1}} \bff(t,\bfx)  \dt\dx,\ \text{for a.e.}~\bfx \in D_\edge, \edge \in \edgesi.
\]
Let us briefly account for the existence of a solution at each step of this algorithm.

First remark that the discrete no slip boundary condition \eqref{scheme:pred_ext} and  \eqref{scheme:cor_ext} are equivalent to requiring that the $i$-th component $u_i$ of the approximate predicted and corrected velocities belongs to the space $\HmeshiNzero =\Bigl\{v\in\Hmeshi(\Omega),\ v(\bfx)=0\ \text{for a.e.}~ \bfx\in D_\edge,\ \text{for any}~ \edge \in \edgesexti\}$. 
We then set $\HmeshNzero(\Omega)=\prod_{i=1}^3 \HmeshiNzero(\Omega)$ and
 $\bfE_{\! N}(\Omega) = \{\bfv \in \HmeshNzero(\Omega) \, : \, \dive_{\! N}\bfv = 0\}$. 
 (See \cite[Section 2]{gal-18-convmac} for the definition of the discrete MAC divergence $\dive_{\! N}$.)
Thanks to the discrete duality of the divergence and gradient operators \cite[Lemma 2.4]{gal-18-convmac}, the space $\bfE_{\! N}(\Omega)$ may also be defined as  $\bfE_{\! N}(\Omega) = \{\bfv \in \HmeshNzero(\Omega) \, : \, \int_\Omega  \bfv \cdot \nabla_{\! N}  w \dx= 0, \; \forall  w \in L_N(\Omega)\}$.

Note that since $\bfu_0 \in \bfE(\Omega)$, we also have $\bfu^0_N \in \bfE_{\! N}(\Omega)$.

\noindent \textit{Prediction step} --
The existence of the predicted velocity follows from Lemma \ref{lem:existpredis}.

\smallskip
\noindent \textit{Correction step} -- 
A weak form of the correction step \eqref{scheme:cor_int}  which computes a divergence-free velocity and an associated pressure reads
\begin{subequations}\label{eq:semi-weakcordis}
\begin{align}
 & \psi_N^{n+1}= p_N^{n+1} - p_N^n \in L_N(\Omega), \int_\Omega \psi_N^{n+1} \dx = 0, \label{eq:semi-weakcordis1} \\  
 &\int_\Omega \nabla_{\! N} \psi_N^{n+1} \cdot \nabla_{\! N} q \dx = \deltat_{\! N} \int_\Omega \tilde \bfu_N^{n+1}\cdot \nabla_{\! N} q \dx, \; \text{for any}~ q \in L_N(\Omega),\label{eq:semi-weakcordis2}\\
 &\bfu_N^{\nalgo+1} =  \tilde \bfu_N^{\nalgo+1} - \frac 1 {\deltat_{\! N}} \nabla_{\! N}  \psi_N^{n+1}. \label{eq:semi-weakcordis3} 
\end{align}
\end{subequations}
Note that if $\bfu_N^{\nalgo+1}$ satisfies \eqref{eq:semi-weakcordis}, then $\int_\Omega  \bfu_N^{\nalgo+1}\cdot \nabla_{\! N} q \dx = 0$ for any $q \in L_N(\Omega)$, so that $ \bfu^{n+1} \in \bfE_{N}(\Omega)$.
The existence of $(\bfu^{\nalgo+1},p^{\nalgo+1}) \in \bfE_{\! N}(\Omega) \times L_N(\Omega)$ satisfying \eqref{eq:semi-weakcordis}  is a consequence of the decomposition result of Lemma \ref{lem:decompdis} (given in the appendix).

We may then define the approximate solutions as follows.

\begin{definition}[Approximate solutions, discrete case] \label{def:exis-dis}
Under the assumptions \eqref{hyp:f-u0} and  \eqref{hyp:timespacedisc}, 
there exists $(\tilde \bfu_N^{\nalgo},\bfu_N^{\nalgo},p_N^{\nalgo})_{\nalgo \in \llbracket 1,N \rrbracket} \subset \HmeshNzero(\Omega)\times \bfE_{\! N}(\Omega) \times L_N(\Omega)$ satisfying \eqref{scheme:pred_ext}-\eqref{schemediv}.
The approximate corrected and predicted velocities may thus be defined by  $ \bfu_N : (0,T) \to \bfE_{\! N}(\Omega)$ and $\tilde \bfu : (0,T) \to \HmeshNzero(\Omega)$ defined by
\begin{equation}\label{eqdef:fullfunctions-dis}
\bfu_\nstep (t) = \sum_{n=0}^{\nstep-1} \ \mathds 1_{(t^\nalgo_\nstep,t^{\nalgo+1}_\nstep]}(t)\bfu^{\nalgo}_\nstep,\quad \quad
\tilde \bfu_\nstep (t) = \sum_{n=0}^{\nstep-1} \  \mathds 1_{(t^\nalgo_\nstep,t^{n+1}_\nstep]}(t) \tilde \bfu^{\nalgo+1}_\nstep,
\end{equation}
\end{definition}
 
For a given $\nstep \ge 1$ and the associated (uniform) time discretization  
\begin{equation}
 \label{eqdef:time-discdis}
 \deltat_\nstep = \frac 1 \nstep, \qquad t_\nstep^\nalgo = n \deltat_\nstep, \; n \in \llbracket 0,N\rrbracket,
\end{equation}

 \begin{remark}[On the boundary conditions] \label{rem:bc1-dis}
The original homogeneous Dirichlet boundary conditions \eqref{boundary} of the strong formulation \eqref{pb:cont} is not imposed by the space $\HmeshNzero(\Omega)$, which only imposes the no slip condition. 
However, it is imposed on the predicted velocity in \eqref{scheme:pred_ext} by the definition of the discrete Laplace operator, see (8)-(10) in \cite[Section 2]{gal-18-convmac}. 
As in the semi-discrete case, it is not imposed in the correction step \eqref{scheme:pmean0}-\eqref{scheme:cor_int}. 

Note also that, as in the semi-discrete case, there is no need for a boundary condition on the pressure in the correction step. 
In fact, it can be inferred from the correction step  that the incremental pressure $\psi^{n+1} =  p^{n+1} - p^\nalgo$ satisfies a discrete Poisson equation on $\Omega$ with a Neumann boundary condition on the boundary. 
\end{remark}

\begin{remark}[On the initial condition]
  If the initial condition $\bfu_0 \in \bfE(\Omega)$ is relaxed to $\bfu_0 \in L^2(\Omega)^3$ as in Remark \ref{rem:semid-init}, the discrete initial condition should be taken as the orthogonal projection onto $\bfE_{\! N}(\Omega)$ of the function $\bfu^0$ defined by \eqref{scheme:init}.
\end{remark}

\begin{remark}
Summing ($\ref{scheme:pred_int}$) and ($\ref{scheme:cor_int}$), we get the discrete equivalent of \eqref{eq:discretepre}:  
\begin{multline}\label{eq:discretepredis}
\frac{1}{\deltat_{\! N}} (\tilde \bfu_N^{n+1} - \tilde \bfu_N^n) +\bfC_{\! N}(\tilde \bfu_N^{n+1})\bfu^n + \nabla_{\!N} (2p_N^n -p_N^{n-1}) \\- \Delta_N \tilde \bfu_N^{n+1} =  \bff_{\! N}^{n+1}\!,  \, \nalgo \in \llbracket 1,N-1 \rrbracket
\end{multline}
\end{remark}

Let us now state the convergence of the algorithm \eqref{eq:scheme} as the time step $\deltat_{\! N}$ and the mesh step $h_N$ tend to 0 (or $N = \frac T {\deltat_{\! N}}\to +\infty$); the proof of this result is the object of the following sections. 

\begin{theorem}[Convergence of the fully discrete projection algorithm] \label{theo:conv-dis}
Under the assumption \eqref{hyp:f-u0}, let $(\deltat_{\! N},\mathcal{D}_N)$ be a sequence of time space discretizations satisfying \eqref{hyp:timespacedisc}, such that $h_N \to 0$ as $N\to + \infty$  and such that the mesh regularity parameter $\theta_N$ defined by \eqref{eqdef:thetaN} remains bounded. 
Let  $ \bfu_N : (0,T) \to \bfE_{\! N}(\Omega)$ and $\tilde \bfu_N : (0,T) \to \HmeshNzero(\Omega)$ be the approximate predicted and corrected velocities defined by the scheme \eqref{eq:scheme} and Definition \ref{def:exis-dis}.
Then there exists $\bar \bfu \in L^2(0,T;\bfE(\Omega))$ $\cap$ $L^\infty(0,T;L^2(\Omega)^3)$ such that up to a subsequence, 

 \begin{itemize}
 \item[] $\tilde \bfu_N\to  \bar \bfu$ in $L^2(0,T;L^2(\Omega)^3)$ as $N\to +\infty$,  
 \item[]$ \gradi_{\! N} \tilde \bfu_N \to \gradi \bar \bfu$ weakly in $L^2((0,T)\times \Omega)^{3 \times 3}$.
 \item[]
 \item[]  $\bfu_N \to \bar \bfu$  in $L^2(0,T;L^2(\Omega)^3)$ and  $\star$-weakly in $L^\infty(0,T;L^2(\Omega)^3)$ as $N\to +\infty$. 
 \end{itemize}
 
 Moreover the function $\bar \bfu$ is a weak solution to \eqref{pb:cont} in the sense of Definition \ref{def:weaksol}.
\end{theorem}
\begin{proof}
 We give here the main steps of the proof, which follows that of the semi-discrete case; these steps are detailed in the following paragraphs.
 \begin{itemize}
  \item \emph{Step 1: first estimates and weak convergence (detailed in Section \ref{sec:estimates}).} 
  Let us define, for $q \in \xN^\ast,$ a discrete $W^{1,q}_0(\Omega)^3$-norm  for the discrete velocity fields.
  For $\bfv \in \Hmeshzero(\Omega)$ with values $(v_\sigma)_{\sigma \in \edges}$ let
\begin{equation}\label{eqdef:W1pnorm}
 \| \bfv \|_{1,q,N}^q = \sum_{i=1}^3 \sum_{\edged = \edge | \edge' \in \edgesdinti} |\edged| \frac{| v_\sigma - v_{\sigma'} |^q}{ d_\edged^{q-1}} + \sum_{i=1}^3 \sum_{\edged \in \edgesdexti. \cap \edgesd(D_\edge)} |\edged| \frac{| v_\sigma  |^q}{ d_\edged^{q-1}}.
\end{equation}
From the energy estimates of Lemma \ref{lem:estdis} below, we get that the approximate velocities $(\tilde \bfu_N)_{N \ge 1}$ and $(\bfu_N)_{N \ge 1}$ given in Definition \ref{def:exis-dis} satisfy
\begin{align} \label{eq:est-dis-tildeu}
&\sup_{\nstep \ge 1}\|\tilde \bfu_\nstep\|_{L^2(0,T;\bfH^1_{0,N}(\Omega))} \leq \cter{cste:estdis},\\
\label{eq:est-dis-u}
&\sup_{\nstep \ge 1}\|\bfu_\nstep\|_{L^{\infty}(0,T;L^{2}(\Omega)^3)} \leq \cter{cste:estdis},\\
&\label{diffdisutu}
\sup_{\nstep \ge 1} \| \bfu_\nstep - \tilde \bfu_\nstep \|_{L^2(0,T;L^2(\Omega)^3)}
  \le \cter{cste:estdis} \deltat_{\! N},
\end{align} 
where
\[\begin{array}{l}
 \displaystyle \hspace{10ex}
\|\bfv\|_{L^2(0,T;\bfH^1_{0,N}(\Omega))}^2 = \sum_{n=0}^{N-1} \deltat\ \|\bfv^{n+1}\|_{1,2,N}^2,
\\[3ex] \displaystyle \hspace{10ex}
\|\bfv\|_{L^\infty(0,T;L^2(\Omega)^3)} = \max \Bigl\{ \|\bfv^{n+1}\|_{L^2(\Omega)^3},\ n\in \llbracket 0, N-1\rrbracket \Bigr\}.
\end{array}
\]
and $\|\cdot\|_{1,2,N}$  is the discrete $H^1_0$ norm defined by \eqref{eqdef:W1pnorm} with $p=2$. 

In particular, \eqref{diffdisutu} yields that
\begin{equation}
 \label{eq:differencedis-u-utilde}
 \bfu_\nstep - \tilde \bfu_\nstep \to 0  \mbox{ in } L^2(0,T;L^2(\Omega)^3)  ~\text{as}~ \Nti.
\end{equation}
 
Owing to \eqref{eq:est-dis-tildeu}-\eqref{eq:est-dis-u}, there exist subsequences still denoted $(\bfu_\nstep)_{\nstep \ge 1}$  and $(\tilde \bfu_\nstep)_{\nstep \ge 1}$ that converge $\star$-weakly in $L^\infty(0,T;L^2(\Omega)^3)$ and weakly in $L^2(0,T;$ $L^2(\Omega)^3)$ respectively.
Moreover, again owing the bound \eqref{eq:est-dis-tildeu} and invoking the compactness result \cite[Theorem 3.1]{eym-10-sto}, there exists a subsequence still denoted by $(\tilde \bfu_{N})_{N \ge 1} $ that converges in $L^2(\Omega)^3$ to a function $\bfu \in H_0^1(\Omega)^3$, and such that $ (\gradi \tilde \bfu_{N})_{N \ge 1} $ converges to $\gradi \bar \bfu$  weakly in $L^2(\Omega)^3$.
By \eqref{eq:differencedis-u-utilde}, the subsequences $(\bfu_\nstep)_{\nstep \ge 1}$ and $(\tilde \bfu_\nstep)_{\nstep \ge 1}$  converge to the same limit $\bar \bfu$ weakly in $L^2(0,T;L^2(\Omega)^3)$.
From the bound \eqref{eq:est-dis-tildeu}, a classical regularity result (see e.g. \cite[Remark 14.1]{eym-00-book}) yields that $\bar \bfu \in L^2(0,T;(H^1_0(\Omega))^3)$.
Passing to the limit in the mass equation (e.g. by a straightforward adaptation of the first step of the proof of \cite[Theorem 3.13]{gal-18-convmac}), it follows that $\bar \bfu \in L^{\infty}(0,T;L^{2}(\Omega)^{3})\cap L^2(0,T;\bfE(\Omega))$.
 
There remains to show that $\bar \bfu$ is a weak solution in the sense of Definition \ref{def:weaksol} and in particular that $\bar \bfu $ satisfies \eqref{weaksol}.
The weak convergence is not sufficient to pass to the limit in the scheme, because of the nonlinear convection term, so that  we first need to get some compactness on one of the subsequences $(\tilde \bfu_N)_{N\in \xN}$ or $(\bfu_N)_{N\in \xN}$ (since, by \eqref{diffutu}, their difference tends to 0 in the $L^2$ norm).
  
\item \emph{Step 2: compactness and convergence in $L^2$ (detailed in section \ref{subsec:dis-compactness})} 
We adapt Step 2 of the convergence proof of the semi-discrete case. 
  Using the bound \eqref{eq:est-dis-tildeu} on the sequence $(\tilde \bfu_\nstep)_{\nstep \ge 1}$, some estimate on the discrete time derivative would be sufficient to obtain the convergence in $L^2(0,T;H_0^1(\Omega)^3)$ by a Kolmogorov-like theorem. 
  As in the semi-discrete case, a difficulty arises from the presence of the (discrete) pressure gradient in Equation \eqref{eq:discretepre}; we get rid of it by multiplying this latter equation by a discrete divergence-free function, chosen as the interpolate of a regular function $\bfvarphi \in L^2(0,T;(W_0^{1,3}(\Omega))^3)$ such that $\dive \bfvarphi = 0$. 
Let us then define the  discrete equivalent of the semi-norm \eqref{etoile-un} on $\HmeshNzero(\Omega)$ by:  
\begin{equation}\label{etoile-un-d}
|\bfw|_{\ast,1,N} = \sup\{ \int_\Omega \bfw \cdot \bfv \dx,~ \bfv \in \bfE_{\! N}(\Omega),~ \Vert \bfv \Vert_{1,3,N} =1\}.
\end{equation}
Estimates on the $L^2(|\cdot|_{\ast,1,N})$ semi-norm of the time translates of the predicted velocity $\widetilde \bfu_N$ are then obtained from the discrete momentum equation \eqref{eq:discretepre}: see Lemma \ref{lem:transdis}.
Again, this is only an intermediate result since we seek an estimate on the time translates of the predicted velocity in the $L^2(L^2)$ norm. 
So next, as in the semi-discrete case, we introduce the discrete equivalent of the semi-norm $|\cdot|_{\ast,0,N}$.
\begin{equation}\label{etoile-zero-d}
 \forall\bfw \in \HmeshNzero(\Omega), \; |\bfw|_{\ast,0,N} = \sup\{ \int_\Omega \bfw \cdot \bfv \dx,~ \bfv \in \bfE_{\! N}(\Omega),~ \Vert \bfv \Vert_{L^2(\Omega)^3} =1\}.
\end{equation}
 
Note that we have the following identity, which is the discrete equivalent of \eqref{ast0-PV}. 
\begin{equation}\label{ast0N-PE}
|\bfw|_{\ast,0,N} = \|\mathcal{P}_{\bfE_{\! N}(\Omega)} \bfw \|_{L^2(\Omega)^3},~\text{for any}~\bfw \in \HmeshNzero(\Omega),
\end{equation}
where $\mathcal{P}_{\bfE_{\! N}(\Omega)}$ is the orthogonal projection operator onto $\bfE_{\! N}(\Omega)$.
Then, thanks to a discrete equivalent of the Lions-like \ref{lem:lions} lemma (Lemma \ref{lem:lionsdis} below), we get that for any $\varepsilon >0$, there exists $C_\varepsilon \in \xR_+$ such that 
\begin{equation}\label{ineq:lionsdis}
    \forall N \in \xN, \; \forall \bfw \in \HmeshNzero, |\bfw|_{\ast,0,N} \le \varepsilon \Vert \bfw\Vert_{1,2,N} + C_{\varepsilon} |\bfw|_{\ast,1,N}.
\end{equation}
From this latter inequality, using Lemma \ref{lem:transdis} on the time translates of $\widetilde \bfu_N$ for the $L^2(|\cdot|_{\ast,1})$ semi-norm and the bound \eqref{eq:est-dis-tildeu}, we get that the time translates of $\widetilde \bfu_N$ for the $L^2(|\cdot|_{\ast,0,N})$ semi-norm also tend to 0 as $N \to +\infty$.

In order to show that the $L^2(L^2)$ norm of the time translates of $\tilde \bfu_N$ tend to 0, we remark that if $\bfv \in \bfE_{\! N}(\Omega)$, then $|\bfv|_{\ast,0,N} = ||\bfv||_{L^2(\Omega)}$ and conclude thanks to \eqref{diffdisutu}, see Lemma \ref{lem:trans-utildedis}).

\item \emph{Step 3: convergence towards the weak solution (detailed in Section \ref{sec:dis-weaksol})} 
Owing to a discrete Aubin-Simon-type theorem \cite[Theoreme 4.53]{gh-edp}, the estimates of steps 1 and 2 yield that there exist subsequences, still denoted $(\bfu_\nstep)_{\nstep \ge 1}$ and $(\tilde \bfu_\nstep)_{\nstep \ge 1}$, that converge to $\bar \bfu$ in $L^2(0,T;L^2(\Omega)^3)$.
Passing to the limit in the scheme \eqref{eq:scheme} then yields that $\bar \bfu$ satisfies ($\ref{weaksol}$) and in particular that $\bar \bfu$ is a weak solution to \eqref{pb:cont}. 
\end{itemize}
\end{proof}

\begin{remark}[Uniqueness and convergence of the whole sequence] \label{rem-uniq-instatdis}
If the solution of the continuous problem is unique, then again the whole sequence converges.
\end{remark}

\subsection{Energy estimates and weak convergence} \label{sec:estimates}

We first obtain a discrete equivalent of the $L^2(0,T;H_0^1(\Omega)^3)$ and $L^\infty(0,T;L^2(\Omega)^3)$ estimates for the predicted and corrected velocity.

\begin{lemma}[Energy estimates] \label{lem:estdis}
Under the assumption \eqref{hyp:f-u0}, let $N \ge 1$, $(\deltat_{\! N},\mathcal{D}_N)$ be a sequence of time space discretization satisfying \eqref{hyp:timespacedisc} and let $(\tilde \bfu_N^{\nalgo},\bfu_N^{\nalgo},p_N^{n})_{n \in \llbracket 0,N \rrbracket}$ $ \subset  \HmeshNzero(\Omega) \times \bfE_{\! N}(\Omega) \times L_N(\Omega)$ be a solution to \eqref{eq:scheme}.
The following estimate holds for $ n \in \llbracket 0,N-1 \rrbracket$:
\begin{multline}\label{eq:local_u_estdis}
  \frac 1 {2 \deltat_{\! N}} \left(\Vert\bfu_N^{\nalgo+1}\Vert_{L^2(\Omega)^3}^2 - \Vert\bfu_N^{\nalgo}\Vert_{L^2(\Omega)^3}^2\right)  \\ + \frac{\deltat_{\! N}}{2} \left( \| \nabla_{\! N}p_N^{n+1} \|_{L^2(\Omega)^3}^2 -  \| \nabla_{\! N}p_N^{n} \|_{L^2(\Omega)^3}^2\right)  \\ +   \frac 1 {2 \deltat_{\! N}} \Vert\tilde\bfu_N^{\nalgo+1} - \bfu_N^{\nalgo}\Vert_{L^2(\Omega)^3}^2 + \| \tilde \bfu_N^{\nalgo+1}\|_{1,2,N}^2 \le \int_\Omega   \bff_{\! N}^{n+1} \cdot \tilde \bfu_N^{n+1} \dx.
\end{multline}
Consequently, there exists $\ctel{cste:estdis}$ depending only on $\Omega$,  $\Vert u_0 \Vert_{{L^2(\Omega)^3}}$,  $\Vert \bff \Vert_{{L^2(\Omega)^3}}$ and $\theta_N$, in a nondecreasing way, such that the estimates \eqref{eq:est-dis-tildeu}-
\eqref{diffdisutu} hold.
\end{lemma}
\begin{proof}
By Lemma \ref{lem:existpredis} with $\alpha = \frac{1}{ \deltat_{\! N}}$, we have for $ n \in \llbracket 0,N-1 \rrbracket$
\begin{multline*}
\frac 1 {2\deltat_{\! N}} \Vert \tilde \bfu_N^{\nalgo+1}\Vert_{L^2(\Omega)^3}^2  - \frac 1 {2\deltat_{\! N}} \Vert \bfu_N^\nalgo\Vert_{L^2(\Omega)^3}^2  + \frac 1 {2\deltat_{\! N}}  \Vert \tilde \bfu_N^{\nalgo+1}- \bfu_N^\nalgo\Vert_{L^2(\Omega)^3}^2  \\ +  \| \tilde \bfu_N^{\nalgo+1}\|_{1,2,N}^2 - \int_\Omega p_N^n \dive_{\! N}\tilde \bfu_N^{n+1} \dx
\le \int_\Omega \bff_{\! N}^{n+1} \cdot \tilde \bfu_N^{n+1} \dx.
\end{multline*}
Squaring the relation \eqref{scheme:cor_int}, integrating over $\Omega$,  multiplying by $\frac{\deltat_{\! N}}{ 2}$ and owing to \eqref{schemediv} and to the discrete duality property of the MAC scheme \cite[Lemma 2.4]{gal-18-convmac}, we get 
\begin{multline*}
  \frac 1 {2 \deltat_{\! N}} \Vert \bfu_N^{\nalgo+1}\Vert_{L^2(\Omega)^3}^2 + \frac{\deltat_{\! N}}{2}\| \nabla_{\! N} p_N^{n+1} \|_{L^2(\Omega)^3}^2 = \frac 1 {2 \deltat_{\! N}} \Vert \tilde \bfu_N^{\nalgo+1}\Vert_{L^2(\Omega)^3}^2 \\ + \frac{\deltat_{\! N}}{2}\| \nabla_{\! N}p_N^n \|_{L^2(\Omega)^3}^2 - \int_\Omega p_N^n \dive_{\! N}\tilde \bfu_N^{n+1} \dx.
\end{multline*}
Summing this latter relation with the previous relation yields for $ n \in \llbracket 0,N-1 \rrbracket$
\begin{multline*}
  \frac 1 {2 \deltat_{\! N}} \left(\Vert\bfu_N^{\nalgo+1}\Vert_{L^2(\Omega)^3}^2 - \Vert\bfu_N^{\nalgo}\Vert_{L^2(\Omega)^3}^2\right) + \frac{\deltat_{\! N}}{2} \left( \| \nabla_{\! N} p_N^{n+1} \|_{L^2(\Omega)^3}^2 -  \| \nabla_{\! N} p_N^{n} \|_{L^2(\Omega)^3}^2\right)   \\  +   \frac{1}{2 \deltat_{\! N}} \Vert\tilde\bfu_N^{\nalgo+1} - \bfu_N^{\nalgo}\Vert_{L^2(\Omega)^3}^2 +\|\tilde \bfu_N^{\nalgo+1}\|_{1,2,N}^2   \le  \int_\Omega  \bff_{\! N}^{n+1} \cdot \tilde \bfu_N^{n+1} \dx.
\end{multline*}

We then get the relation \eqref{eq:local_u_estdis} using the Cauchy-Schwarz inequality and the discrete Poincar\'e estimate \cite[Lemma 9.1]{eym-00-book} after summing over the time steps.
\end{proof}

\subsection{Estimates on the time translates and compactness}  \label{subsec:dis-compactness}

\begin{lemma}[A first estimate on the time translates]\label{lem:transdis}
Under the assumptions of Theorem \ref{theo:conv-dis}, there exists $\ctel{cste:transdis} >0$ only depending on $|\Omega|$,  the $L^2$-norm of $\bfu_0$ and the $L^2$-norm of $\bff$ such that for any $N \ge 1$ and for any $\tau \in (0,T)$
\begin{equation*}
\int_0^{T-\tau} |\tilde \bfu_N(t+\tau,\cdot) -  \tilde \bfu_N(t,\cdot) |^2_{\ast,1,N} \dt \le \cter{cste:transdis}  \tau (\tau +\deltat_{\! N}).
\end{equation*}
\end{lemma}

\begin{proof}
For $t \in (0,T-\tau)$,  
$
\displaystyle\tilde \bfu_N(t+\tau)  -\tilde \bfu_N(t)  = \sum_{n=1}^{N-1} \chi_{N,\tau}^n(t) (\tilde \bfu_N^{n+1} - \tilde \bfu_N^n)$, with $\chi_{N,\tau}^n$ defined by \eqref{eqdef:chi-N-tau}.
Using \eqref{eq:discretepredis}, we thus get that
\begin{multline*}
\tilde \bfu_N(t+\tau) - \tilde \bfu_N(t) = \deltat_{\! N} \!\! \sum_{n=1}^{N-1} \!\!   \chi_{N,\tau}^n(t)\Delta_N \tilde \bfu_N^{n+1}  - \deltat_{\! N} \sum_{n=1}^{N-1} \!\!   \chi_{N,\tau}^n(t) \bfC_{N}(\tilde \bfu_N^{n+1}) \bfu_N^n\\
 - \deltat_{\! N} \sum_{n=1}^{N-1} \!\!   \chi_{N,\tau}^n(t) \nabla_{\! N} (2 p_N^{n} - p_N^{n-1}) + \deltat_{\! N} \sum_{n=1}^{N-1} \!\!   \chi_{N,\tau}^n(t) \bff_{\! N}^{n+1}.
\end{multline*}
 
Let $\bfvarphi  \in \bfE_{\! N}(\Omega)$ and let $A(t)= \int_\Omega  \bigl( \tilde\bfu_N(t+\tau) - \tilde \bfu_N(t)\bigr) \cdot \bfvarphi \dx.$
Since 
\[\displaystyle \int_\Omega\chi_{N,\tau}^n(t) \Delta_N \tilde \bfu_N^{n+1} \cdot \bfvarphi\dx = \int_\Omega\chi_{N,\tau}^n(t) \gradi_{\! N} \tilde \bfu_N^{n+1}: \gradi_{\! N} \bfvarphi\dx,\] 
where $\gradi_{\! N}$ is the gradient operator of the velocity defined on each dual rectangular grid, see \cite[Section 2]{gal-18-convmac}), we get that
\begin{align*}
&A(t)=  A_{d}(t) + A_{c}(t) + A_{p}(t) + A_{\bff}(t) \mbox{ with}\\
&
A_{d}(t) = -\sum_{n=1}^{N-1}    \chi_{N,\tau}^n(t)\deltat_{\! N} \int_\Omega \gradi_{\! N} \tilde{\bfu}_N^{n+1} : \gradi_{\! N} \bfvarphi \dx,
\\ &
A_{c}(t) = - \sum_{n=1}^{N-1}    \chi_{N,\tau}^n(t)\deltat_{\! N}  b_{N}(\tilde \bfu_N^{n+1}),\bfu_N^n, \bfvarphi)
\\ &
A_{p}(t) = \sum_{n=1}^{N-1}    \chi_{N,\tau}^n(t) \deltat_{\! N} \int_\Omega  (2 p_N^{n} - p_N^{n-1}) \dive_{\! N}\bfvarphi \dx,
\\ &
A_{\bff}(t) =  \sum_{n=1}^{N-1}    \chi_{N,\tau}^n(t)\deltat_{\! N} \int_\Omega   \bff_{\! N}^{n+1}  \cdot  \bfvarphi \dx,
\end{align*}
with
\[
b_{N}(\tilde \bfu_N^{n+1}),\bfu_N^n, \bfvarphi)= \bfC_{N}(\tilde \bfu_N^{n+1})\bfu_N^n \cdot \bfvarphi.
\]
Let us reproduce at the fully discrete level the computations done for each of these terms in the proof of Lemma \ref{lem:transsemidis}.

By a technique similar to that of  \cite[Lemma 3.5]{gal-18-convmac}, we get that
\begin{equation}\label{eq:est_Addis}
A_{d}(t) \le |\Omega|^{1/6} \| \bfvarphi \|_{1,3,N} \sum_{n=1}^{N-1} \chi_{N,\tau}^n(t) \deltat_{\! N} \| \tilde{\bfu}_N^{n+1}\|_{1,2,N}.
\end{equation}

Using H\"older's inequality with exponents 2, 6 and 3 ($\frac 1 2 + \frac 1 6 + \frac 1 3 = 1$), we get (similarly to the estimate of \cite[Lemma 3.5]{gal-18-convmac}) that there exists $\ctel{cste:conv_forms}$ such that
\begin{equation*} 
A_{c}(t)  \le \cter{cste:conv_forms} \sum_{n=1}^{N-1}   \chi_{N,\tau}^n(t) \deltat_{\! N} \|\bfu_N^{n} \|_{L^2(\Omega)^3} \|  \tilde{\bfu}_N^{n+1} \|_{L^6(\Omega)^3} \|  \bfvarphi \|_{1,3,N}.
\end{equation*}
 By the discrete Sobolev inequality \cite[Lemma 9.1]{eym-00-book}, there exists $\ctel{cste:sobdis2}\in \xR_+$ depending only on $|\Omega|$ and $\theta_N$ in a nondecreasing way such that (see \cite[Lemma 3.5]{eym-00-book})
\begin{equation*}
\| \bfv \|_{L^6(\Omega)^3} \le \cter{cste:sobdis2} \| \bfv\|_{1,2,\edges},~\text{for any}~\bfv \in \HmeshNzero(\Omega).
\end{equation*}
Therefore, thanks to the boundedness assumptions on $\widetilde \bfu_\nstep$ and $\bfu_\nstep$,
\begin{equation}\label{eq:est_Acdis}
A_{c}(t) \le \cter{cste:est} \cter{cste:sobdis2}  \cter{cste:conv_forms} \|  \bfvarphi \|_{1,3,N}\sum_{n=1}^{N-1}   \chi_{N,\tau}^n(t) \deltat_{\! N} \| \tilde{\bfu}_N^{n+1}  \|_{1,2,N},
\end{equation}
 
Again invoking the discrete Sobolev inequality, there exists $\ctel{cste:sobdis1}\in \xR_+$ only depending on  $|\Omega|$ such that
\begin{equation*}
\| \bfv \|_{L^3(\Omega)^3} \le \cter{cste:sobdis1} \| \bfv\|_{1,3,N},~\text{for any}~\bfv \in \HmeshNzero(\Omega).
\end{equation*}
Consequently, \begin{equation}\label{eq:est_Afdis}
A_{\bff}(t) \le \cter{cste:sobdis1} |\Omega|^{1/6}\| \bfvarphi \|_{1,3,N}  \sum_{n=1}^{N-1}   \chi_{N,\tau}^n(t)\deltat_{\! N} \| \bff_{\! N}^{n+1} \|_{L^2(\Omega)^3}.
\end{equation}
 
 Thanks to the fact that $\bfvarphi \in \bfE_{\! N}(\Omega)$ and to the discrete duality property stated in \cite[Lemma 2.4]{gal-18-convmac}, $A_p(t) =0$.

Summing Equations \eqref{eq:est_Addis}, \eqref{eq:est_Afdis}, \eqref{eq:est_Acdis} we obtain
$$
A(t) \le C  \| \bfvarphi \|_{1,3,N} \sum_{n=1}^{N-1}   \chi_{N,\tau}^n(t) \deltat_{\! N} ( \| \tilde \bfu_N^{n+1} \|_{1,2,N} + \|  \bff_{\! N}^{n+1} \|_{L^2(\Omega)^3})
$$
where $C = |\Omega|^{1/6} +\cter{cste:sobdis1}|\Omega|^{1/6}  +\cter{cste:est} \cter{cste:sobdis2}  \cter{cste:conv_forms}$.
This implies 
$$
|\tilde  \bfu_N(t+\tau) -  \tilde \bfu_N(t) |_{\ast,1,N} \le C \sum_{n=1}^{N-1}  \chi_{N,\tau}^n(t) \deltat_{\! N}( \| \tilde \bfu_N^{n+1} \|_{1,2,N} + \| \bff_{\! N}^{n+1} \|_{L^2(\Omega)^3}).
$$
Using the fact that $\sum_{n=1}^{N-1}  \chi_{N,\tau}^n(t) \deltat_{\! N} \le \tau +\deltat_{\! N} $ for any $t \in (0,T-\tau)$ we then obtain
$$
|\tilde  \bfu(t+\tau) -  \tilde \bfu(t) |_{\ast,1,N}^2 \le 2 C^2 (\tau +\deltat_{\! N})  \sum_{n=1}^{N-1} \chi_{N,\tau}^n(t) \deltat_{\! N} ( \| \tilde \bfu_N^{n+1} \|_{1,2,N}^2 + \|\bff_{\! N}^{n+1}\|_{L^2(\Omega)^3}^2).
$$
Using the fact that $ \int_0^{T-\tau} \chi_{N,\tau}^n(t) \dt \le \tau $ for any $n \in \llbracket 1,N-1 \rrbracket$ we obtain
\begin{multline*}
\int_0^{T-\tau} |\tilde  \bfu_N(t+\tau) -  \tilde \bfu_N(t) |_{\ast,1,N}^2 \dt  \\  \le 2 C^2 (\tau +\deltat_{\! N})  \sum_{n=1}^{N-1} \deltat_{\! N} ( \| \tilde \bfu_N^{n+1} \|_{1,2,N}^2 + \|\bff_{\! N}^{n+1} \|_{L^2(\Omega)^3}^2) \int_0^{T-\tau}  \chi_{N,\tau}^n(t) \dt \\
\le   2 C^2 (\tau +\deltat_{\! N}) \tau ( \| \tilde \bfu_N \|_{L^2(0,T:\HmeshNzero(\Omega))}^2 + \| \bff \|_{L^2((0,T) \times \Omega)^3}^2) 
\le \cter{cste:esttrans1} \tau (\tau +\deltat_{\! N})
\end{multline*}
which gives the expected result.
\end{proof}

For $\bfv =(v_1,v_2,v_3) \in (L^2(\Omega)^3$,  we define $\widetilde{\mathcal{P}}_N \bfv$ as the vector function with piecewise constant components:  the $i$-th component of $\widetilde{\mathcal{P}}_N \bfv$ is constant on each dual cell $D_\edge$, $\edge \in \edges$, and equal to the mean value of $v_i$ on the face $\edge$.  
By \cite[Lemma 3.7]{gal-18-convmac}, $\widetilde{\mathcal{P}}_N$ is a Fortin operator in the sense that it preserves the divergence; in particular,  
\[  
    \bfv \in \bfE(\Omega)\Longrightarrow \widetilde{\mathcal P}_N \bfv \in \bfE_{\! N}(\Omega).
\]

\begin{lemma}[Lions-like, fully discrete version]
\label{lem:lionsdis}
Consider a  rectangular domain $\Omega$ of $\xR^3$ and a sequence of MAC grids $(\mathcal{D}_N)_{N \ge 1}$ of $\Omega$ satisfying \eqref{hyp:mesh} such that $h_N \to 0$ as $N\to + \infty$  and such that the mesh regularity parameter $\theta_N$ defined by \eqref{eqdef:thetaN} remains bounded.
Then, for any $\varepsilon >0$, there exists $C_{\varepsilon} > 0$ and $N_\varepsilon \ge 1$ depending on $\varepsilon$ such that for any $N \ge N_\varepsilon$ and for any $\bfw \in \HmeshNzero(\Omega)$, \eqref{ineq:lionsdis} is satisfied.
\end{lemma}
\begin{proof}
Let $\varepsilon >0$; let us show by contradiction that there exists $C_{\varepsilon} >0$ and $N_\varepsilon \ge 1$ depending on $\varepsilon$ such that for any $N \ge N_\varepsilon$ and for any $\bfw \in \HmeshNzero(\Omega)$
\begin{equation*}
   | \bfw |_{\ast,0,N} \le \varepsilon \Vert \bfw\Vert_{1,2,N} + C_{\varepsilon} |\bfw |_{\ast,1,N}.
\end{equation*}
 
Suppose that this is not so, then there exist $\varepsilon >0$ and a subsequence of MAC grids of $\Omega$ still denoted by $(\mathcal{D}_N)_{N \ge 1}$ and a sequence  $(\bfw_N)_{N \ge 1}$ of functions such that $\bfw_N \in \HmeshNzero(\Omega)$ for any  $N \ge 1$ and, thanks to \eqref{ast0N-PE},
\begin{equation*}
     \| \mathcal{P}_{\bfE_{\! N}(\Omega)} \bfw_N \|_{L^2(\Omega)^3}  = | \bfw_N |_{\ast,0,N} > \varepsilon \Vert \bfw_{N} \Vert_{1,2,N} + N |\bfw_{N}|_{\ast,1,N},~\text{for any}~N \ge 1.
\end{equation*}
By a homogeneity argument, we may choose $\Vert  \mathcal{P}_{\bfE_{\! N}(\Omega)} \bfw_N\Vert_{L^2(\Omega)^3} = 1$; it then follows from the latter inequality that the sequence $(\| \bfw_{N} \|_{1,2,N})_{N \ge 1}$ is bounded and that $|\bfw_{N}|_{\ast,1,N} \to 0$ as $\Nti$.
Hence there exists a subsequence still denoted by $ (\bfw_{N})_{N \ge 1} $ that converges in $L^2(\Omega)^3$ to a function $\bfw \in H_0^1(\Omega)^3$, see e.g. \cite[Theorem 3.1]{eym-10-sto}.
Lemma \ref{lem:contprojedges} given below then yields that $\mathcal{P}_{\bfE_{\! N}(\Omega)} \bfw_N \to \mathcal{P}_{\bfV(\Omega)} \bfw$  in  $L^2(\Omega)^3$ and in particular $\| \mathcal{P}_{\bfV(\Omega)} \bfw \|_{L^2(\Omega)^3}=1$. 
(Recall that $\mathcal{P}_{\bfV(\Omega)} : L^2(\Omega)^3 \to L^2(\Omega)^3$ is the orthogonal projection in $L^2(\Omega)^3$ onto the space $\bfV(\Omega)$.)

For any $\bfvarphi \in \bfW(\Omega)$, we have $\widetilde{\mathcal P}_{N}(\bfvarphi) \in \bfE_{\! N}(\Omega)$.
Since $\bfw_N - \mathcal{P}_{\bfE_{\! N}(\Omega)} \bfw_N \perp \bfE_{\! N}$ and by definition of $|\bfw_N|_{\ast,1,N}$, it follows that 
\[
  \int_\Omega \mathcal{P}_{\bfE_{\! N}(\Omega)} \bfw_N \cdot \widetilde{\mathcal P}_{N}(\bfvarphi) =  \int_\Omega  \bfw_N \cdot \widetilde{\mathcal P}_{N}(\bfvarphi) \dx \le |\bfw_N|_{\ast,1,N}  |\widetilde{\mathcal P}_{N}(\bfvarphi)|_{1,3,N}.
  \]
By the $W^{1,q}$ stability of the operator $\widetilde{\mathcal P}_{N}$ stated in \cite[Theorem 1]{gal-12-w1q}, there exists $\ctel{cste:pedgessigma}$ only depending on $|\Omega|$ and on $\theta_N$ in a nondecreasing way, such that
\[
  \int_\Omega \mathcal{P}_{\bfE_{\! N}(\Omega)} \bfw_N \cdot \widetilde{\mathcal P}_{N}(\bfvarphi) =  \int_\Omega  \bfw_N \cdot \widetilde{\mathcal P}_{N}(\bfvarphi) \dx \le \cter{cste:pedgessigma} |\bfw_N|_{\ast,1,N} \Vert \bfvarphi \Vert_{W^{1,3}_0(\Omega)^3}.
\]
\begin{equation*}
\| \widetilde{\mathcal{P}}_N \bfvarphi \|_{1,3,N} \le \cter{cste:pedgessigma} \| \bfvarphi \|_{W^{1,3}_0(\Omega)^3},~\text{for any}~\bfvarphi \in W^{1,3}_0(\Omega)^3.
\end{equation*}
 
Passing to the limit in this inequality yields that 
\[
  \int_\Omega \mathcal{P}_{\bfV(\Omega)} \bfw \cdot \bfvarphi \dx = 0,~\text{for any}~ \bfvarphi \in \bfW(\Omega).
\]
This in turn implies that there exists $\xi \in H^1(\Omega)$ such that $ \mathcal{P}_{\bfV(\Omega)} \bfw= \gradi \xi$. Using the fact that $ \mathcal{P}_{\bfV(\Omega)} \bfw \in \bfV(\Omega)$ we have
$$
\|  \mathcal{P}_{\bfV(\Omega)} \bfw \|_{L^2(\Omega)^3}^2 = \int_\Omega \mathcal{P}_{\bfV(\Omega)} \bfw \cdot \gradi \xi \dx = 0,
$$ 
which contradicts $\Vert  \mathcal{P}_{\bfV(\Omega)} \bfw \Vert_{L^2(\Omega)^3} = 1$.
\end{proof}

\begin{lemma}\label{lem:contprojedges}
Let  $N \ge 1$ and let $\mathcal{D}_N=(\mesh_N,\edges_N)$ be a MAC grid of $\Omega$ in the sense of \eqref{hyp:mesh}, such that $(h_{N})_{\nstep \ge 1}$ converges to zero and such that  $(\theta_{N})_{\nstep \ge 1}$ is bounded, with $\theta_N$ defined by \eqref{eqdef:thetaN}. 
Let $(\bfv_N)_{N \ge 1}$ be a sequence of functions such that $\bfv_N \in \HmeshNzero(\Omega)$ for any $N \ge 1$ and $(\bfv_N)_{N\ge 1}$ converges to $\bfv $ in $L^2(\Omega)^3$. 
Then the sequence $(\mathcal{P}_{\bfE_{\! N}(\Omega)}\bfv_N)_{N \ge 1}$ converges to $\mathcal{P}_{\bfV(\Omega)} \bfv$ in $L^2(\Omega)^3$.
\end{lemma}
\begin{proof}
Using the fact that $(\bfv_N)_{N\ge 1}$  is bounded in $L^2(\Omega)^3$ we obtain that the sequence $(\mathcal{P}_{\bfE_{\! N}(\Omega)}\bfv_N)_{N \ge 1}$ is bounded in $L^2(\Omega)^3$. 
Hence there exists a subsequence still denoted by $(\mathcal{P}_{\bfE_{\! N}(\Omega)}\bfv_N)_{N \ge 1}$ that  converges to a function $\tilde \bfv$ weakly in $L^2(\Omega)^3$. 
Thanks to the discrete duality property stated in \cite[Lemma 2.4]{gal-18-convmac}, we have, for any $\varphi \in C_c^\infty(\xR^3)$, 
\[
\int_\Omega \mathcal{P}_{\bfE_{\! N}(\Omega)}\bfv_N \cdot \nabla_{\! N} \Pi_N \varphi \dx = 0,~\text{for any}~N \ge 1,
\]
where $\Pi_N\varphi$ is the piecewise constant function defined by $\Pi_N   \varphi (\bfx)= \dfrac 1 {|K|}\displaystyle \int_K \varphi \dx$ for all $\bfx \in K$, $K \in \mesh_N$.
The discrete gradient $\nabla_{\! N}$ is consistent in the sense of \cite[Lemme 2.3]{gal-18-convmac} and therefore there exists $\ctel{consgrad} \in \xR_+$ depending only on $\Omega$ and on $\theta_N$ in a nondecreasing way, such that 
$$
\left| \int_\Omega \mathcal{P}_{\bfE_{\! N}(\Omega)}\bfv_N \cdot \gradi  \varphi \dx \right| \le \cter{consgrad} h_N  \| \mathcal{P}_{\bfE_{\! N}(\Omega)}\bfv_N \|_{L^2(\Omega)^3} \| \gradi^2 \varphi \|_{L^\infty(\Omega)^{3 \times 3}},~\text{for any}~N \ge 1.
$$
Passing to the limit in the previous identity gives
$$
\int_\Omega \tilde \bfv \cdot \gradi \varphi \dx=0,~\text{for any}~\varphi \in C_c^\infty(\xR^3).
$$
We then obtain that $\tilde \bfv \in \bfV(\Omega)$.
Since $\widetilde{\mathcal{P}}_{N}$ preserves the divergence \cite[Lemma 3.7]{gal-18-convmac}, the following identity holds for any $ \bfvarphi \in \bfV(\Omega) \cap C_c^1(\Omega)^3$ 
$$
  \int_\Omega \bfv_N \cdot \widetilde{\mathcal{P}}_{N} \bfvarphi \dx  =\int_\Omega \mathcal{P}_{\bfE_{\! N}(\Omega)}\bfv_N \cdot \widetilde{\mathcal{P}}_{N} \bfvarphi \dx,~\text{for any}~N \ge 1.
$$
Passing to the limit in the previous identity gives
$$
\int_\Omega \bfv \cdot \bfvarphi \dx = \int_\Omega \tilde \bfv \cdot \bfvarphi \dx~\text{for any}~\bfvarphi \in \bfV(\Omega).
$$
We then obtain that $ \tilde \bfv = \mathcal{P}_{\bfV(\Omega)} \bfv$ and the sequence $(\mathcal{P}_{\bfE_{\! N}(\Omega)}\bfv_N)_{N \ge 1}$ converges to $\mathcal{P}_{\bfV(\Omega)} \bfv$ weakly in $L^2(\Omega)^3$. We can write
$$
\| \mathcal{P}_{\bfE_{\! N}(\Omega)}\bfv_N \|_{L^2(\Omega)^3}^2
= \int_\Omega \bfv_N \cdot \mathcal{P}_{\bfE_{\! N}(\Omega)}\bfv_N \dx,~\text{for any}~N \ge 1.
$$
Using the convergence of the sequence $(\bfv_N)_{N \ge 1}$ to $\bfv$ in $L^2(\Omega)^3$ and the weak convergence of the sequence $(\mathcal{P}_{\bfE_{\! N}(\Omega)}\bfv_N)_{N \ge 1}$ to $\mathcal{P}_{\bfV(\Omega)} \bfv$ in $L^2(\Omega)^3$ we obtain
$$
\lim_{\Nti} \| \mathcal{P}_{\bfE_{\! N}(\Omega)}\bfv_N \|_{L^2(\Omega)^3}^2 = \int_\Omega \bfv \cdot \mathcal{P}_{\bfV(\Omega)} \bfv \dx = \| \mathcal{P}_{\bfV(\Omega)} \bfv \|_{L^2(\Omega)^3}^2.
$$
The weak convergence of the sequence $(\mathcal{P}_{\bfE_{\! N}(\Omega)}\bfv_N)_{N \ge 1}$ to $\mathcal{P}_{\bfV(\Omega)}$ in $L^2(\Omega)^3$  and convergence of the sequence $( \| \mathcal{P}_{\bfE_{\! N}(\Omega)}\bfv_N \|_{L^2(\Omega)^3})_{N \ge 1}$ to $ \| \mathcal{P}_{\bfV(\Omega)} \bfv \|_{L^2(\Omega)^3}$ gives the expected result.
\end{proof}

Since the predicted velocities are bounded in the $\Vert \cdot \Vert_{1,2,N}$ norm (see Lemma \ref{lem:estdis}), their $|\cdot|_{\ast,0,N}$ semi-norm is controlled by their $|\cdot|_{\ast,1,N}$ semi-norm thanks to Lemma \ref{lem:lionsdis}.  
As in Lemma \ref{lem:trans-utildesemidis}, we can therefore obtain an estimate on the time translates for the $L^2(0,T;L^2(\Omega)^3)$ norm, and, as a consequence, the $L^2(0,T;L^2(\Omega)^3)$ convergence of the predicted velocities.

\begin{lemma}\label{lem:trans-utildedis}
Under the assumptions of Theorem \ref{theo:conv-dis} the sequence $(\tilde \bfu_{N})_{N \ge 1}$ satisfies 
\[ 
    \int_0^{T-\tau} \Vert \tilde \bfu_{N}(t+\tau) - \tilde \bfu_{N}(t) \Vert_{L^2(\Omega)^3}^2 \dt \to 0 \mbox{ as } \tau \to 0,
\]
uniformly with respect to $N$, and is therefore relatively compact in $L^2(0,T;L^2(\Omega)^3)$.
\end{lemma}
\begin{proof}
We follow the proof of Lemma \ref{lem:trans-utildesemidis}.
By the triangle inequality, 
 \begin{align*}
  &\int_0^{T-\tau} \Vert \tilde \bfu_\nstep(t+\tau) - \tilde \bfu_\nstep(t) \Vert_2^2 \dt \le 2(A_\nstep(\tau) + B_\nstep(\tau)), \mbox{ with } \\
  & A_\nstep(\tau) =  \int_0^{T-\tau} \Vert (\tilde \bfu_\nstep -\bfu_\nstep)(t+\tau) - (\tilde \bfu_\nstep-\bfu_\nstep)(t) \Vert_2^2 \dt, \\
  & B_\nstep(\tau) =  \int_0^{T-\tau} \Vert   \bfu_\nstep(t+\tau) - \bfu_\nstep(t) \Vert_2^2 \dt. 
 \end{align*}
For any $N$,  $A_\nstep(\tau) \to 0$ as $\tau \to 0$, but owing to \eqref{eq:differencedis-u-utilde}, we get that $A_\nstep(\tau) \to 0$ as $\tau \to 0$, uniformly with respect to $N$.
Let us prove that this is also the case for $B_\nstep(\tau) \to 0$.

Since $\bfu_\nstep(t) \in \bfE_{\! N}(\Omega)$ for any $t \in (0,T)$ we have for any $t \in (0,T-\tau)$ 
\begin{align*}\Vert  \bfu_\nstep(t+\tau) -  \bfu_\nstep(t) \Vert_{L^2(\Omega)^3} &=\sup_{\substack{\bfv \in  \bfE_{\! N}(\Omega) \\  \Vert \bfv\Vert_{L^2(\Omega)^3} = 1}}\int_\Omega \left(\bfu_\nstep(t+\tau) -  \bfu_\nstep(t)\right) \cdot \bfv \dx  \\
 &\le \Vert  (\bfu_\nstep - \tilde \bfu_\nstep)(t+\tau) -  (\bfu_\nstep - \tilde \bfu_\nstep)(t) \Vert_{L^2(\Omega)^3} + \\ &\hspace{.2cm}\sup_{\substack{\bfv \in  \bfE_{\! N}(\Omega) \\  \Vert \bfv\Vert_{L^2(\Omega)^3} = 1}}
\int_\Omega (\tilde \bfu_\nstep(t+\tau) -  \tilde \bfu_\nstep(t)) \cdot \bfv \dx,
  \end{align*}
so that 
\[
B_N(\tau)  \le 2 A_\nstep(\tau) +  2 \int_0^{T-\tau} |\tilde  \bfu_\nstep(t+\tau) -  \tilde \bfu_\nstep(t) |_{\ast,0,N}^2 \dt.
\]

Now thanks to Lemma \ref{lem:lionsdis}, for any $\varepsilon >0$, there exists $C_\varepsilon  >0$ and $N_\varepsilon \ge 1$ such that for any $\nstep \ge N_\varepsilon$ and for any $t \in (0,T-\tau)$
\begin{multline*}
 |\tilde  \bfu_\nstep(t+\tau) -  \tilde \bfu_\nstep(t) |_{\ast,0,N} \le \varepsilon \Vert \tilde  \bfu_\nstep(t+\tau) -  \tilde \bfu_\nstep(t) \Vert_{1,2,N} \\+ C_\varepsilon |\tilde  \bfu_\nstep(t+\tau) -  \tilde \bfu_\nstep(t) |_{\ast,1,N}, 
\end{multline*}
In particular for any $N \ge N_\varepsilon $ and for any $\tau \in (0,T)$ we have
\begin{multline*}
 \int_0^{T-\tau} |\tilde  \bfu_\nstep(t+\tau) -  \tilde \bfu_\nstep(t) |^2_{\ast,0,N} \dt  \le 2 \varepsilon^2 \int_0^{T-\tau} \Vert \tilde  \bfu_\nstep(t+\tau) -  \tilde \bfu_\nstep(t) \Vert_{1,2,N}^2 \dt \\ + 2C_\varepsilon^2 \int_0^{T-\tau} |\tilde  \bfu_\nstep(t+\tau) -  \tilde \bfu_\nstep(t) |^2_{\ast,1,N} \dt.
\end{multline*}

Therefore, owing to lemmas \ref{lem:estdis} and \ref{lem:transdis}, for any $ \nstep \ge N_\varepsilon$ and for any $\tau \in (0,T)$

\begin{equation*}
\int_0^{T-\tau} |\tilde  \bfu_\nstep(t+\tau) -  \tilde \bfu_\nstep(t) |^2_{\ast,0,N} \dt 
\le 8 \cter{cste:estdis}^2 \varepsilon^2  + 2C_\varepsilon^2  \cter{cste:transdis} \tau (\tau+\deltat_{\! N}).
\end{equation*}
Hence, for any $ \nstep \ge N_\varepsilon$ and for any $\tau \in (0,T)$,
\[ 
B_\nstep(\tau) \le 2 A_\nstep(\tau) +  16 \cter{cste:estdis} \varepsilon^2 + 4 C_\varepsilon^2 \cter{cste:transdis} \tau (\tau+\deltat_{\! N}).
\]
 Let $\zeta > 0$ be given, and let:
\begin{itemize}
 \item  $\tau_0 > 0$ such that for any $\tau \in (0,\tau_0)$, $2 A_\nstep(\tau) \le \zeta$ for any $N \ge 1$.
 \item $\varepsilon > 0$ such that $ 16 \cter{cste:estdis}^2 \varepsilon^2 \le \zeta$,
\item  $\tilde \tau_0 > 0$ such that $ 2 C_{\varepsilon}^2 \cter{cste:transdis} \tau (\tau+\deltat_{\! N}) \le \zeta$ for any $ \tau \in (0,\tilde \tau_0) $ and $N \ge 1$.
\end{itemize}
We then obtain that $B_\nstep(\tau) \le 3 \zeta$ for any $ \tau \in (0, \min(\tau_0,\tilde \tau_0)) $ and  $N \ge N_{\varepsilon}$. 
Using the fact that $B_\nstep(\tau) \to 0$ as $\tau \to 0$ for any $N \ge 1$ we obtain that $B_\nstep(\tau)  \to 0$  as $\tau \to 0$, uniformly with respect to $N$. 
The proof of Lemma \ref{lem:trans-utildedis} is thus complete.
\end{proof}

\subsection{Convergence towards the weak solution} \label{sec:dis-weaksol}
Now that we have proven that the approximate velocities $(\tilde \bfu_N)_{N \ge 1}$ and $(\bfu_N)_{N \ge 1}$ converge
in $L^2(0,T;L^2(\Omega)^3)$, up to a subsequence, to a common limit $\bar \bfu\in L^2(0,T;H^1_0(\Omega)^3)$, there remains to show, as in the semi-discrete case, that $\bar \bfu$ is a weak solution to \eqref{pb:cont} in the sense of Definition \ref{def:weaksol}.

\begin{lemma}[Lax-Wendroff consistency of the discrete scheme]\label{lem:weaksoldis}
Let $(\tilde \bfu_N)_{N \ge 1}$ and $(\bfu_N)_{N \ge 1}$ $ \subset L^2(0,T;L^2(\Omega)^3)$ be sequences of solutions to the fully discrete scheme \eqref{eq:scheme} for $N \in \xN$ (see Definition \ref{def:exis-dis}), and assume that $\bar \bfu \in (L^2(0,T;H^1_0(\Omega)^3)$ is such that $\tilde \bfu_N \to \bar \bfu$ in $L^2(0,T;L^2(\Omega)^3)$  and $\tilde \bfu_N \to \bar \bfu$ weakly in $L^2(0,T;L^2(\Omega)^3)$ as $N\to + \infty$, and that the sequence $(\tilde \bfu_N)_{N \ge 1}$ is bounded in the $L^2(\Vert \cdot\Vert_{1,2,N})$ norm. 
Then the function $\bar \bfu$ is a weak solution to \eqref{pb:cont} in the sense of Definition \ref{def:weaksol}.
\end{lemma}

\begin{proof}
Let $\bfvarphi \in C_c^\infty([0,T) \times \Omega)^3$, such that $\dive \bfvarphi =0$ in $\Omega$.
Using ($\ref{scheme:pred_int}$) and ($\ref{scheme:cor_int}$) to obtain ($\ref{eq:discretepredis}$) we have for $n \in \llbracket 1,N-1 \rrbracket$
\begin{equation*}
\frac{1}{\deltat_{\! N}}( \tilde \bfu_N^{n+1} - \tilde \bfu_N^n) + \bfC_{\! N}(\tilde \bfu_N^{n+1})\bfu_N^n + \nabla_{\! N} (2 p_N^n - p_N^{n-1}) - \Delta_{\edges_N} \tilde \bfu_N^{n+1} = \bff_{\! N}^{n+1}.
\end{equation*}
By Lemma \cite[Lemma 3.7]{gal-18-convmac} we have $\dive_{\! N}\, \widetilde{\mathcal P}_N \bfvarphi(t_N^n,\cdot)=0$. 
We set $\bfvarphi_N^n = \widetilde{\mathcal P}_N \bfvarphi(t_N^n,\cdot) \in \bfE_{\! N}(\Omega)$ and multiply the previous identity by $\deltat_{\! N} \bfvarphi_N^{n}$,  integrate over $\Omega$, and sum over $n \in \llbracket 1,N-1 \rrbracket$. 
This yields
\begin{multline}\label{eq:weaksoldisRvarphi}
\sum_{n=1}^{N-1} \int_\Omega ( \tilde \bfu_N^{n+1} - \tilde \bfu_N^n)  \cdot \bfvarphi_N^n \dx \dt
+\sum_{n=1}^{N-1}\deltat_{\! N}  b_N( \bfu_N^n,\tilde \bfu_N^{n+1},\bfvarphi_N^n)
\\
+\sum_{n=1}^{N-1} \deltat_{\! N} \int_\Omega \gradi_{\edgesd_N}\bfu_N^{n+1} : \gradi_{\edgesd_N}\bfvarphi_N^n \dx 
= \sum_{n=1}^{N-1} \deltat_{\! N} \int_\Omega \bff_{\! N}^{n+1} \cdot \bfvarphi_N^n \dx.
\end{multline}
Using the fact that $\bfvarphi_N^{N}=0$ in $\Omega$ the first term of the left hand side reads
\begin{multline*}
 \sum_{n=1}^{N-1} \int_\Omega (\tilde \bfu_N^{n+1} - \tilde \bfu_N^n)  \cdot \bfvarphi_N^n \dx \dt \\
 = - \sum_{n=1}^{N-1} \int_\Omega  \tilde \bfu_N^{n+1}\cdot ( \bfvarphi_N^{n+1} - \bfvarphi_N^n) \dx \dt - \int_\Omega \tilde \bfu_N^1 \cdot \bfvarphi_N^0 \dx. 
\end{multline*}
Since $ \bfvarphi_N^0 = \widetilde{\mathcal{P}}_{N}(\bfvarphi(0,\cdot)) \in \bfE_{\! N}(\Omega) $ and, owing to \eqref{scheme:init}, $\bfu_N^0=  \widetilde{\mathcal{P}}_{N}\bfu_0 $, so that
\begin{multline*}
 \sum_{n=1}^{N-1} \int_\Omega \left(\tilde \bfu_N^{n+1} - \tilde \bfu_N^n\right)  \cdot \bfvarphi_N^n \dx = - \sum_{n=1}^{N-1} \int_\Omega  \tilde \bfu_N^{n+1}\cdot (\bfvarphi_N^{n+1} - \bfvarphi_N^n) \dx \dt \\  - \int_\Omega \widetilde{\mathcal{P}}_{N} \bfu_0 \cdot \widetilde{\mathcal{P}}_N(\bfvarphi(0,\cdot))  \dx  - \int_\Omega  (\tilde \bfu_N^1 -  \bfu_N^0) \cdot \widetilde{\mathcal{P}}_N(\bfvarphi(0,\cdot))  \dx.
\end{multline*}
The regularity  of $\bfvarphi$ implies that
$$
\lim_{\Nti} \sum_{n=1}^{N-1} \dfrac {\bfvarphi_N^{n+1} - \bfvarphi_N^n} {\deltat_{\! N}} \ \Char_{(t_N^\nalgo,t_N^{n+1}]}(\cdot) =  \partial_t \bfvarphi~\text{in}~L^\infty((0,T) \times \Omega)^3.$$
Using the weak convergence of the sequence $(\tilde \bfu_N)_{N \ge 1}$ in $L^2(0,T;L^2(\Omega)^3)$, the uniform convergence of the sequence $(\widetilde{\mathcal{P}}_N(\bfvarphi(0,\cdot)))_{N \ge 1}$, the convergence in $L^2(\Omega)$ of the sequence $(\widetilde{\mathcal{P}}_N \bfu_0)_{N \ge 1}$ and \eqref{diffdisutu}, we obtain
\begin{equation}\label{eq:convweakstep1dis}
\lim_{\Nti}\sum_{n=1}^{N-1} \int_\Omega ( \tilde \bfu_N^{n+1} - \tilde \bfu_N^n)  \cdot \bfvarphi_N^n \dx \dt = -\int_0^T \int_\Omega \bar \bfu \cdot \partial_t \bfvarphi \dx\dt 
- \int_\Omega \bfu_{0} \cdot \bfvarphi(0,\cdot) \dx.
\end{equation}
Finally, the proof that
\begin{equation}\label{eq:convweakstep2dis}
\lim_{\Nti} \sum_{n=1}^{N-1} \deltat_{\! N} \int_\Omega \gradi_{\edgesd_N}\tilde \bfu_N^{n+1} : \gradi_{\edgesd_N}\bfvarphi_N^n \dx = \int_0^T \int_\Omega \gradi \bar \bfu : \gradi \bfvarphi \dx \dt,
\end{equation}
\begin{equation}\label{eq:convweakstep3dis}
\lim_{\Nti} \sum_{n=1}^{N-1} \deltat_{\! N}\, b_N(\bfu_N^n,\tilde \bfu_N^{n+1},\bfvarphi_N^n) \to - \int_0^T \int_\Omega \bar \bfu \otimes \bar \bfu : \gradi \bfvarphi \dx \dt,
\end{equation}
and  
\begin{equation}\label{eq:convweakstep4dis}
\lim_{\Nti} \sum_{n=1}^{N-1} \deltat_{\! N} \int_\Omega \mathcal  \bff_{\! N}^{n+1} \cdot \bfvarphi_N^n \dx \dt
= \int_0^T \int_\Omega \bff \cdot \bfvarphi \dx\dt.
\end{equation}
follows the proof of the convergence of the equivalement terms in the proof of\cite[Theorem 4.3]{gal-18-convmac}.
Using ($\ref{eq:convweakstep1dis}$)-($\ref{eq:convweakstep4dis}$) and passing to the limit in ($\ref{eq:weaksoldisRvarphi}$) gives the expected result.
\end{proof}
 
\appendix

\section{Some technical lemmas}
\begin{lemma}[Existence and estimate for the linearized equation] \label{lem:existpre}
Let $\Omega$ be an open bounded connected subset of $\xR^3$ with Lipschitz boundary.
Let $\alpha >0$ and let $\bfu \in \bfV(\Omega)$, $p \in L^2(\Omega)$ and $\bff \in L^2(\Omega)^3$. 
There  exists $\tilde \bfu \in H_0^1(\Omega)^3$ such that 
\begin{multline}\label{eq:existpre}
\alpha \int_\Omega \tilde \bfu \cdot \bfv \dx -  \int_\Omega \tilde \bfu \otimes \bfu : \gradi \bfv \dx + \int_\Omega \gradi \tilde \bfu : \gradi  \bfv \dx  \\
= \alpha \int_\Omega  \bfu \cdot \bfv \dx + \int_\Omega p \dive \bfv \dx + \int_\Omega \bff \cdot \bfv \dx~\text{for any}~ \bfv \in C_c^1(\Omega)^3.
\end{multline}
Moreover $\tilde \bfu$ satisfies 
\begin{multline}\label{eq:existpreest}
\frac{\alpha}{2} \Vert \tilde \bfu\Vert_{L^2(\Omega)^3}^2 -\frac{\alpha}{2}  \Vert \bfu\Vert_{L^2(\Omega)^3}^2  +  \frac{\alpha}{2} \Vert \tilde \bfu- \bfu\Vert_{L^2(\Omega)^3}^2  \\  - \int_\Omega p \dive \bfu \dx  + \Vert  \tilde \bfu\Vert_{H_0^1(\Omega)^3}^2  \le  \int_\Omega \bff \cdot \tilde \bfu\dx. 
\end{multline}
\end{lemma}
\begin{proof}
Using Lemma \ref{lem:denV} there exists a sequence   $(\bfu_n)_{n \ge 0} $  of functions of $\bfV(\Omega) \cap C_c^1(\Omega)^3 $ converging to $\bfu$  in $L^2(\Omega)^3$. 
Consider the following problem:
\begin{align}
\nonumber
& \mbox{Find }\tilde \bfu_n \in H_0^1(\Omega)^3 \mbox{ such that }\\
\label{eq:weaktildeun}
& \alpha \int_\Omega \tilde \bfu_n \cdot \bfv \dx - \int_\Omega \tilde \bfu_n \otimes \bfu_n : \gradi \bfv + \int_\Omega \gradi \tilde \bfu_n : \gradi \bfv \dx  \\
\nonumber
& \qquad = \alpha \int_\Omega  \bfu_n \cdot \bfv \dx + \int_\Omega p \dive \bfv \dx + \int_\Omega \bff \cdot \bfv \dx,~\text{for any}~ \bfv \in H_0^1(\Omega)^3.
\end{align}
By the Lax-Milgram theorem, there exists a unique $\tilde \bfu_n \in H_0^1(\Omega)^3$ to this problem; indeed, the left hand-side of \eqref{eq:weaktildeun} is a bilinear continuous and coercive form on $H_0^1(\Omega)^3 \times H_0^1(\Omega)^3$ because
\begin{equation}\label{skew}
 \int_\Omega \tilde \bfu \otimes \bfu : \gradi \tilde \bfu \dx = 0,~\text{for any}~(\tilde \bfu,\bfu) \in H_0^1(\Omega)^3 \times \bfE(\Omega);
\end{equation}
moreover the right hand side in ($\ref{eq:weaktildeun}$) is a linear continuous form on $H_0^1(\Omega)^3 $. 

Take $\bfv = \tilde \bfu_n$ in \eqref{eq:weaktildeun}; owing to \eqref{skew}, we get
$$
\alpha \| \tilde \bfu_n \|_{L^2(\Omega)^3}^2 + \|  \tilde \bfu_n \|_{H_0^1(\Omega)^3}^2 = \alpha \int_\Omega  \bfu_n \cdot \tilde \bfu_n \dx + \int_\Omega p \dive \tilde \bfu_n \dx + \int_\Omega \bff \cdot \tilde \bfu_n \dx
$$
which implies 
\begin{multline}\label{eq:weaktildeunest}
\frac{\alpha}{2} \| \tilde \bfu_n \|_{L^2(\Omega)^3}^2  -\frac{\alpha}{2}  \Vert \bfu_n \Vert_{L^2(\Omega)^3}^2  +  \frac{\alpha}{2} \Vert \widetilde \bfu_n- \bfu_n\Vert_{L^2(\Omega)^3}^2 + \|  \tilde \bfu_n \|_{H_0^1(\Omega)^3}^2 \\  =  \int_\Omega p \dive \tilde \bfu_n \dx + \int_\Omega \bff \cdot \tilde \bfu_n,~\text{for any}~n \ge 0.
\end{multline}
Therefore, by the Young inequality, the sequence $(\tilde \bfu_n)_{n \ge 0}$ is bounded in $H_0^1(\Omega)^3$ and in particular passing to a subsequence converges to $\tilde \bfu \in H_0^1(\Omega)^3$  in $L^2(\Omega)^3$ and weakly in $H_0^1(\Omega)^3$.  The convergence in $L^2(\Omega)^3$ of the sequence $(\bfu_n)_{n \ge 0}$ gives
$$
\lim_{\nti}\int_\Omega \tilde \bfu_n \cdot \bfv \dx = \int_\Omega \tilde \bfu \cdot \bfv \dx,~\text{for any}~\bfv \in C_c^1(\Omega).
$$
The weak convergence in $H_0^1(\Omega)^3$ of the sequence $(\tilde \bfu_n)_{n \ge 0}$ gives
$$
\lim_{\nti}\int_\Omega \gradi \tilde \bfu_n : \gradi \bfv \dx = \int_\Omega \gradi \tilde \bfu : \gradi \bfv \dx,~\text{for any}~\bfv \in C_c^1(\Omega).
$$
 The convergence in $L^2(\Omega)^3$ of the sequence $(\bfu_n)_{n \ge 0}$ and the  convergence in $L^2(\Omega)^3$ of the sequence $(\tilde \bfu_n)_{n \ge 0}$ gives

$$
\lim_{\nti}\int_\Omega \tilde \bfu_n \otimes \bfu_n : \gradi \bfv \dx = \int_\Omega \tilde \bfu \otimes \bfu : \gradi \bfv \dx,~\text{for any}~\bfv \in C_c^1(\Omega).
$$
 Passing to the limit in ($\ref{eq:weaktildeun}$) with $\bfv \in C_c^1(\Omega)^3$ gives ($\ref{eq:existpre}$).  The weak convergence in $H_0^1(\Omega)^3$ of the sequence $(\tilde \bfu_n)_{n \ge 0}$ gives 
$$
\|  \tilde \bfu \|_{H_0^1(\Omega)^3}^2 \le \liminf_{n \to \infty} \|  \tilde \bfu_n \|_{H_0^1(\Omega)^3}^2
$$
Passing to the limit  ($\ref{eq:weaktildeunest}$) gives ($\ref{eq:existpreest}$) which concludes the proof of Lemma \ref{lem:existpre}.
\end{proof}

Let us now give the decomposition result which was used for the proof of existence of a solution to the correction step \eqref{eq:semi-weakcor}.

\begin{lemma}[Decomposition of $L^2$ vector fields]\label{lem:decomp}
Let $\Omega$ be an open bounded connected subset of $\xR^3$ with Lipschitz boundary.
Then for any $\bfw \in L^2(\Omega)^3 $ there exists  $(\bfv,\psi) \in \bfV(\Omega) \times H^1(\Omega) $ such that $ \bfw =  \bfv + \gradi \psi$. 
\end{lemma}
\begin{proof}
Let $\psi$ be a solution  (unique, up to a constant) of the problem
\begin{align*}
& \psi \in H^1(\Omega),
\\
& \int_\Omega \nabla\psi \cdot \nabla \xi \dx=  \int_\Omega \bfw \cdot \nabla \xi \dx, \textrm{ for any}~  \xi \in H^1(\Omega).
 \end{align*}
Then $\bfw = \bfv+ \nabla \psi $ with $\int_\Omega \bfv \cdot \nabla \xi \dx =0$ for any $\xi \in H^1(\Omega)$.
\end{proof}

The following lemma gives a characterisation of the gradient which is used in the proof of Lemma \ref{lem:lions}. 
Its proof is a simple consequence of a result of M. E. Bogovskii \cite{bogovskii} and refer to the very clear presentation of \cite{crouzeix} for more on this subject.

\begin{lemma}[Characterization of the gradient]\label{lem:carac-grad}
Let $\Omega$ be an open bounded connected subset of $\xR^3$ with Lipschitz boundary.
Let $\bff \in L^2(\Omega)^3$ such that $\int_\Omega \bff \cdot \bfvarphi \dx = 0$ for all  $\bfvarphi \in C_c^\infty(\Omega)^3$ such that $\dive \bfvarphi=0$ in $\Omega$.
Then there exists $\xi \in L^2(\Omega)$ such that $\bff = \nabla \xi$.
\end{lemma}
\begin{proof}
We recall that  $L^2_0(\Omega)=\{q \in L^2(\Omega)$ such that $\int_\Omega q(x) \dx=0\}$.
A classical result \cite{bogovskii} gives the existence of an  linear continuous operator $\mathcal{B} : L^2_0(\Omega) \to H^1_0(\Omega)^3$
such that $\dive(\mathcal{B}(q))=q$ a.e. in $\Omega$.
Furthermore $\mathcal{B}(\varphi) \in C_c^\infty(\Omega)^3$ for any $\varphi \in C_c^\infty(\Omega) \cap L_0^2(\Omega)$.

For $q \in L^2_0(\Omega)$ we set $T(q)=\int_\Omega \bff  \cdot \mathcal{B}(q)  \dx$. 
The mapping $T$ is a linear continuous form on $L^2_0(\Omega)$.
There exists  $\xi \in L_0^2(\Omega)$ such that
\[
    T(q) = \int_\Omega \bff  \cdot \mathcal{B}(q)  \dx=\int_\Omega \xi  q  \dx,~\text{for any}~q \in L_0^2(\Omega).
\]
Taking now $\bfvarphi \in C_c^\infty(\Omega)^3$, one has $\dive\bfvarphi \in C_c^\infty(\Omega) \cap L^2_0(\Omega)$ so that
$\bfvarphi - \mathcal{B}(\dive \bfvarphi) \in C_c^\infty(\Omega)^3$ and $\dive( \bfvarphi - \mathcal{B}(\dive \bfvarphi) )=0$ in $\Omega$. 
Then, the hypothesis on $\bff$ gives $\int_\Omega \bff \cdot ( \bfvarphi - \mathcal{B}(\dive \bfvarphi)) \dx = 0$ which leads to
\[
\int_\Omega \bff \cdot \bfvarphi \dx=\int_\Omega \bff \cdot \mathcal{B}(\dive \bfvarphi) \dx=\int_\Omega \xi \dive \bfvarphi\dx,
\]
and we conclude $\nabla \xi=\bff$ (that is the distribution $\nabla \xi$ is the function $f$).
\end{proof}

A consequence of this lemma is the following interesting \emph{per se} density result.
\begin{lemma}[Density of divergence-free functions]\label{lem:denV}
Let $\Omega$ be an open bounded connected subset of $\xR^3$ with Lipschitz boundary. Let $\mathcal{V}(\Omega) = \{ \bfvarphi \in C_c^\infty(\Omega)^3~\text{such that}~\dive \bfvarphi = 0~\text{in}~\Omega\}.$ 
The closure of $\mathcal{V}(\Omega)$ in $L^2(\Omega)^3$ is $\bfV(\Omega)$.
\end{lemma}
\begin{proof}
Equipped with the $L^2(\Omega)^3$-norm, the space $\bfV(\Omega)$ is a Hilbert space. 
In order to prove this density result, we prove that, in this Hibert space, $\mathcal{V}(\Omega)^\perp=\{0\}$.

Let $\bfv \in \bfV(\Omega)$ and assume $\bfv \in \mathcal{V}(\Omega)^\perp$.
Then, Lemma \ref{lem:carac-grad} gives the existence of $\xi \in L^2(\Omega)$ such that $\bfv = \nabla \xi$ (and then $\xi \in H^1(\Omega)$).
Since $\bfv \in \bfV(\Omega)$, ones deduces for $\psi \in H^1(\Omega)$
\[
\int_\Omega \nabla \xi \cdot \nabla \psi \dx= \int_\Omega \bfv \cdot \nabla \psi \dx=0.
\]
In particular this gives $\int_\Omega \nabla \xi \cdot \nabla \xi \dx=0$ and then $\bfv=\nabla \xi=0$. 
This proves that $\mathcal{V}(\Omega)$ is dense in $\bfV(\Omega)$.
\end{proof}

\section{Some discrete technical lemmas}

We assume that $\Omega$ is an open rectangular parallelepiped.
In addition, we assume that the edges (respectively the faces) of $\Omega$
are orthogonal to one vector of the canonical basis of $\xR^3$.

\begin{lemma}[Existence and estimate, discrete linearized equation] \label{lem:existpredis}
Let $\disc_N=(\mesh_N,\edges_N)$ be a MAC grid of $\Omega$ indexed by $N \ge 1$, and let $\alpha >0$.
Let $\bfu \in \bfE_{\! N}(\Omega)$, $p \in L_N(\Omega)$ and $\bff_{\! N} \in \HmeshNzero(\Omega)$. 
There  exists a unique   $\tilde \bfu \in\HmeshNzero(\Omega)$ such that 
\begin{multline}\label{eq:weaktildeudiscreteexist}
\alpha \int_\Omega \tilde \bfu \cdot \bfv \dx +b_N(\bfu,\tilde \bfu,\bfv) + \int_\Omega \gradi_N  \tilde \bfu : \gradi_N \bfv \dx  \\
= \alpha \int_\Omega  \bfu \cdot \bfv \dx + \int_\Omega p \dive_{\! N} \bfv \dx + \int_\Omega \bff \cdot \bfv \dx,~\text{for any}~ \bfv \in \HmeshNzero(\Omega).
\end{multline}
Moreover $\tilde \bfu$ satisfies 
\begin{multline}\label{eq:existpreestdis}
\frac{\alpha}{2} \Vert \tilde \bfu\Vert_{L^2(\Omega)^3}^2 -\frac{\alpha}{2}  \Vert \bfu\Vert_{L^2(\Omega)^3}^2  +  \frac{\alpha}{2} \Vert \tilde \bfu- \bfu\Vert_{L^2(\Omega)^3}^2  \\   - \int_\Omega p \dive_{\! N} \tilde \bfu \dx   + \Vert  \tilde \bfu\Vert_{1,2,\edges}^2  \le  \int_\Omega \bff_{\! N} \cdot \tilde \bfu\dx. 
\end{multline}
\end{lemma}
\begin{proof}
The weak formulation of \eqref{eq:weaktildeudiscreteexist} reads:
\begin{multline}\label{eq:semi-weakpreddis}
\alpha \int_\Omega \tilde \bfu^{\nalgo+1} \cdot \bfv \dx  +b_N(\bfu^n,\tilde \bfu^{n+1},\bfv) + \int_\Omega \gradi_N  \tilde \bfu^{n+1} : \gradi_N \bfv \dx \\
= \alpha \int_\Omega  \bfu^\nalgo \cdot \bfv \dx + \int_\Omega p^\nalgo\ \dive_N \bfv \dx + \int_\Omega \bff_{\! N} \cdot \bfv \dx~\text{for any}~ \bfv \in \Hmeshzero(\Omega).
\end{multline}
This formulation is equivalent to the  form \eqref{scheme:pred_int}. 
The existence of a unique $\tilde \bfu \in \HmeshNzero(\Omega)$ satisfying ($\ref{eq:weaktildeudiscreteexist}$) is consequence of the that fact the left hand-side in ($\ref{eq:weaktildeudiscreteexist}$) is a bilinear continuous and coercive form on $\HmeshNzero(\Omega) \times \HmeshNzero(\Omega)$ and the right hand side in ($\ref{eq:weaktildeudiscreteexist}$) is a linear continuous form on $\HmeshNzero(\Omega)$. 
More precisely the left-hand side is coercive as a consequence of Lemma \cite[Lemma 3.6]{gal-18-convmac}.
We take $\bfv = \tilde \bfu$ in ($\ref{eq:weaktildeudiscreteexist}$) and using \cite[Lemma 3.6]{gal-18-convmac}, we obtain 
$$
\alpha \| \tilde \bfu \|_{L^2(\Omega)^3}^2 + \|  \tilde \bfu \|_{1,2,0}^2 \le \alpha \int_\Omega  \bfu \cdot \tilde \bfu \dx + \int_\Omega p \dive_{\! N} \tilde \bfu \dx + \int_\Omega \bff \cdot \tilde \bfu \dx
$$
which implies
\begin{multline*}
\frac{\alpha}{2} \| \tilde \bfu \|_{L^2(\Omega)^3}^2  -\frac{\alpha}{2}  \Vert \bfu \Vert_{L^2(\Omega)^3}^2  +  \frac{\alpha}{2} \Vert \tilde \bfu- \bfu\Vert_{L^2(\Omega)^3}^2  + \|  \tilde \bfu \|_{1,2,\edges}^2  \\  \le \int_\Omega p \dive_{\! N} \tilde \bfu \dx  + \int_\Omega \bff \cdot \tilde \bfu \dx.
\end{multline*}
which gives the expected result.
 \end{proof}

The following lemma is the discrete version of lemma \ref{lem:decomp} which was used for the proof of existence of a solution to the correction step \eqref{eq:semi-weakcordis}.
\begin{lemma}[Decomposition of $\HmeshNzero(\Omega)$ vector fields]\label{lem:decompdis}
Let $\disc=(\mesh,\edges)$ be a MAC grid of $\Omega$.
Then for any $\bfw \in \HmeshNzero(\Omega) $ there exists  $(\bfv,\psi) \in \bfE_{\! N}(\Omega) \times L_N(\Omega) $ such that $ \bfw =\bfv+ \nabla_{\! N} \psi $. %
\end{lemma}
\begin{proof}
Let $\psi$ be a solution  (unique, up to a constant)  of  
\begin{align*}
& \psi \in L_N(\Omega),
\\
& \int_\Omega \nabla_{\! N} \psi \cdot \nabla_{\! N} \xi \dx=  \int_\Omega \bfw \cdot \nabla_{\! N} \xi \dx, \textrm{ for any }  \xi \in L_N(\Omega).
 \end{align*}
Then $\bfw =\bfv+ \nabla_{\! N} \psi $ with $\bfv \in \bfE_{\! N}(\Omega)$.
\end{proof}

\bibliographystyle{abbrvurl}
\bibliography{nsproj}
\end{document}